\documentclass[letterpaper, 11pt,  reqno]{amsart}
\usepackage[margin=1.2in,marginparwidth=1.5cm, marginparsep=0.5cm]{geometry}

\usepackage{amsmath,amssymb,amscd,amsthm,amsxtra, esint, xcolor,
amsfonts, mathtools, mathrsfs}

\usepackage[utf8]{inputenc}

\usepackage[implicit=true]{hyperref}
\hypersetup{pdftitle={Functional second-order Gaussian Poincare inequalities},
pdfauthor={Anna Vidotto and Guangqu Zheng}}
\usepackage{enumitem}
\setlist[itemize]{leftmargin=5.5mm}

\DeclareMathOperator{\tr}{tr}

\allowdisplaybreaks[2]

\usepackage{color}

\newtheorem{theorem}{Theorem} [section]

\newtheorem{lemma}[theorem]{Lemma}
\newtheorem{proposition}[theorem]{Proposition}
\newtheorem{remark}[theorem]{Remark}

\newtheorem{definition}[theorem]{Definition}

\newcommand{\I}{\hspace{0.5mm}\text{I}\hspace{0.5mm}}
\newcommand{\II}{\text{I \hspace{-2.8mm} I} }
\newcommand{\III}{\text{I \hspace{-2.9mm} I \hspace{-2.9mm} I}}

\newcommand{\dom}{\textup{dom} }

\newcommand{\noi}{\noindent}

\newcommand{\R}{\mathbb{R}}

\newcommand{\NN}{\mathcal{N}}

\newcommand{\bul}{\bullet}

\DeclareMathOperator{\Law}{Law}

\newcommand{\E}{\mathbb{E}}

\newcommand{\Cov}{\textup{Cov}}

\newcommand{\dl}{\delta}

\newcommand{\nb}{\nabla}

\newcommand{\s}{\sigma}

\newcommand{\cF}{{\mathcal{F}}}
\newcommand{\cB}{{\mathcal{B}}}

\newcommand{\wt}{\widetilde}

\makeatletter
\newcommand{\vast}{\bBigg@{3.5}}
\newcommand{\Vast}{\bBigg@{5}}
\makeatother

\renewcommand{\(}{\left(}
\renewcommand{\)}{\right)}
\renewcommand{\[}{\left[}
\renewcommand{\]}{\right]}
\renewcommand{\{}{\left\lbrace}
\renewcommand{\}}{\right\rbrace}

\renewcommand{\o}{\omega}
\renewcommand{\O}{\Omega}

\newcommand{\jb}[1]
{\langle #1 \rangle}

\newcommand{\abs}[1]
{\left| #1 \right|}

\newcommand{\ind}{\mathbf 1}

\renewcommand{\P}{\mathbb{P}}

\newcommand{\N}{\mathbb{N}}

\newtheorem*{ackno}{Acknowledgements}

\newcommand{\fH}{\mathfrak{H}}

\newcommand{\Id}{\textup{Id}}

\numberwithin{equation}{section}
\numberwithin{theorem}{section}

\begin{document}
\baselineskip = 14pt

\title[Functional 2nd-order Gaussian Poincar\'e inequalities]
{Functional second-order \\ Gaussian Poincar\'e inequalities}

 \author[A. Vidotto and G. Zheng]
{Anna Vidotto and Guangqu Zheng}

\address{
Anna Vidotto,
Dipartimento di Scienze di Base e Applicate per l'Ingegneria (SBAI)\\
Sapienza Universit\`a Di Roma\\
Via Antonio Scarpa 16\\
00161 Roma (RM), Italy
}

\email{anna.vidotto@uniroma1.it}

\address{
Guangqu Zheng,
Department of Mathematics and Statistics\\
Boston University\\
665 Commonwealth Avenue\\
Boston, MA 02215\\
USA
 }

\email{gzheng90@bu.edu}



\begin{abstract}

In this paper, we work in the framework of Hilbert-valued
Wiener structures and derive a functional version of
the second-order Gaussian Poincar\'e inequality that
leads to abstract bounds for Gaussian process approximation
in $d_2$ distance used by Bourguin and Campese (Electron. J. Probab., 2020).
As pointed out by Bassetti, Bourguin, Campese, and Peccati
(Statist. Probab. Lett., 2026)
 this distance alone does not imply weak convergence
in infinite dimensions, we also consider $d_2^\dagger$,
obtained by adding the absolute difference of the second moments.
This metric metrizes $2$-Wasserstein convergence
and agrees with $d_2$ for centered variables
with matching covariance operators.
Our abstract bounds are flexible
and can be applied in various examples including
 functional Breuer-Major
central limit theorems, shallow neural networks,
and spatial statistics of SPDE solutions.
In these applications, the covariance estimates also give
the second-moment convergence needed for weak convergence
in the corresponding Hilbert spaces.

\end{abstract}

\date{\today}
%
\maketitle
%


\baselineskip = 14pt

\section{Introduction}\label{SEC1}

This paper is devoted to the study of quantitative functional limit
theorems for Hilbert-valued   random variables.
Let us first motivate the study by focusing on the one-dimensional case.
The classical central limit  theorem (CLT) asserts that
for a sequence of real-valued, identically,
and independently distributed random variables
$\{X_i: i\geq 1\}$ with mean zero and variance one,
$V_n:= \frac{1}{\sqrt{n}} (X_1 + ... + X_n)$ converges in law
to   $\NN(0,1)$ as $n\to+\infty$.
Under the additional assumption $\E[ |X_1|^3] <\infty$,
the following quantitative central limit theorem, known as
the Berry-Esseen bound, holds true:
with $Z\sim\NN(0,1)$,
\begin{align}\label{Kol}
d_{\rm Kol}( V_n,  Z ) := \sup_{t\in\R}  |\P(V_n \leq t) - \P(Z\leq t)|
\leq  \frac{C\,\E[ |X_1|^3]}{\sqrt{n}},
\end{align}
where $d_{\rm Kol}$ is the Kolmogorov distance
and $C > 0$ is an absolute constant; see, e.g.,
\cite[Section 3.7]{NP12}.
The above  random variable $V_n$  is a linear statistic
over independent random variables, however such limit theorems
have been developed very well both for the cases
when the random variables $X_i$ are (at least weakly) dependent
and when one considers highly nonlinear
statistics over independent random variables
(e.g.,  multi-linear polynomials in i.i.d$.$
standard normal random variables).
A powerful tool for establishing quantitative central limit theorems
is Stein's method (e.g., \cite[Section 1]{Cha08} and \cite{CGS}).
Moreover, when the non-linear statistics are constructed from a Gaussian
process,
the obtention of limit theorems is possible by combining
Stein's method, Malliavin calculus, and Gaussian analysis,
as developed in \cite{NP09} and later in the monograph \cite{NP12}.
Relevant to our work in this paper,
one of the  important  developments in Stein's method is
the second-order Gaussian Poincar\'e inequality that was first proposed
by S.~Chatterjee \cite{Cha09}.
Consider the nonlinear statistic $F = g(G_1, ... ,G_n)$ in i.i.d. $\NN(0,1)$
random variables with the (deterministic) function $g:\R^n\to \R$
twice differentiable, the Gaussian Poincar\'e inequality assets
that the variance of $F$ is controlled by
$\E [ \| \nabla  g(G_1, ... ,G_n)\|^2_{\R^n}  ]$,
that is, when $\nabla  g(G_1, ... ,G_n)$ is ``typically'' small,
the fluctuation of $F$ is small (first-order result).
Chatterjee investigated a possible second-order result in his work
\cite{Cha09}. More precisely, when the operator norm
of the  Hessian matrix
$\nabla^2  g(G_1, ... ,G_n)$ is typically small compared to the variance of
$F$,
$F$ is close to a normal distribution with matched mean and variance,
with the proximity measured in the total-variation distance. Soon later,
the paper \cite{NPR09} by
Nourdin, Peccati, and Reinert
generalized Chatterjee's result
to the case where $F$ may depend on infinitely many
Gaussian random variables or $F$ is a general functional
over a Gaussian process. In a paper \cite{Vid20} by the first author,
an improved version of the second-order Gaussian
 Poincar\'e inequality  has been derived and often
 it yields optimal rate of convergence in applications;
  see \cite[Sections 4.1 and 4.2]{Vid20}.
Given the flexibility of such a tool, we have seen applications of the
second-order Gaussian
Poincar\'e inequality and its variants in the context of
random matrix theory \cite{Cha09, Vid20}, and, more recently, in
Gaussian neural networks with a fixed input  \cite{BFF24}
and  spatial statistics arising from  stochastic partial differential
equations \cite{BNQSZ,NXZ22},
to name a few.

 The goal of the present paper is to extend such inequalities
 to random variables taking values in Hilbert spaces.
We closely follow the framework developed in
  \cite{BC20} by Bourguin and Campese.
The authors of    \cite{BC20}
 combined   an infinite-dimensional version of Stein's method
\cite{Sh11} and the so-called Gamma calculus
to derive  various Malliavin-Stein bounds for functional approximation
in the $d_2$ metric \eqref{d2}. Note that the approach to Stein's method
for Hilbert-valued random variables is different from the one
in the seminal paper \cite{adb90} by A.D. Barbour;
see discussions in \cite[Section 1.1]{BRZ24}.
In the current article, we aim at refining the aforementioned  Malliavin-Stein
bounds
in the spirit of \cite{Vid20}, i.e., by adapting
the second-order Gaussian Poincar\'e inequalities
to the Hilbert-valued Wiener structures.
Our main abstract bounds are stated in Theorem \ref{thm_imp1} and
Theorem \ref{thm_imp2}.
To illustrate these bounds, we provide three applications with different
contexts
in Theorem \ref{thm_app1} (functional Breuer-Major CLT),
Theorem \ref{thm_NN} (shallow neural networks),
and Theorem \ref{APP_SPDE} (stochastic partial differential  equations).

There is an important distinction between approximation in $d_2$
and weak convergence in the Hilbert space.
Bassetti, Bourguin, Campese, and Peccati \cite{BBCP26}
showed that the Hilbert-Schmidt Hessian condition
\eqref{HSH1}
 used to define
$d_2$ does not suffice to metrize weak convergence
in infinite dimensions.
We retain this distance for the second-order estimates
and supplement it with the absolute difference
of the second moments, obtaining the metric $d_2^\dagger$
\eqref{d2dag}.
The passage to weak convergence uses the standard
finite-dimensional approximation argument and the
Wasserstein moment criterion
\cite[Definition 6.8 and Theorem 6.9]{Vil09};
see Proposition \ref{prop_d2mom} for more details.
For related moment criteria in $L^p$ spaces, see
Cremers and Kadelka \cite{CK86}.
We also refer to D\"uker and Zoubouloglou \cite{DZ25}
for Hilbert-valued fourth-moment theorems
and Malliavin-Stein bounds involving trace-class norms.
Our second-order bounds yield the same rates in $d_2^\dagger$
when the covariance operators match.
The additional diagonal covariance estimates in the applications
then give convergence in $2$-Wasserstein distance,
and hence weak convergence, in the stated Hilbert spaces.

 \smallskip

\noi
{$\bul$ \bf Structure of this paper.}
Section \ref{SEC2} is devoted to preliminaries:
in Section \ref{SEC21}, we set up our framework,
notably the Hilbert-valued Wiener structure from
\cite{BC20}; in Section \ref{SEC22},
we take a pedestrian approach to
 define various Malliavin operators
for Hilbert-valued random variables by using chaos expansion;
in Section \ref{SEC2_3}, we state the abstract Malliavin-Stein bounds,
together with random contraction inequality in the Hilbert setting.
With these preliminaries,
we present our main abstract bounds together with     applications
 in Section \ref{SEC3}.

\section{Malliavin-Stein bounds for the Hilbert-valued Wiener
structure}\label{SEC2}

\subsection{Hilbert-valued  Wiener structures}  \label{SEC21}

Let  $\big(K, \langle \, \cdot , \, \cdot \rangle_K \big)$
be a  real separable Hilbert space
equipped with the Borel $\s$-algebra $\cB(K)$,
and let $(\Omega, \cF, \P)$ denote the probability space,
on which all the random objects in this paper are defined.
A $K$-valued random variable $X$ is simply a measurable
function from $\Omega$ to $K$, i.e.,
$X^{-1}(B) \in \cF$
for any $B\in \cB(K)$.
It can be characterized by
the real-valued random variables
$\{ \langle X, v \rangle_K : v\in K \}$.
In particular, fixing an orthonormal basis $\{ k_i: i\in\N_{\geq 1} \}$
of $K$,\footnote{The choice of the orthonormal basis is immaterial,
for example, the definition of trace \eqref{tr} does not depend
on such a choice. We will not emphasize this point throughout
this paper. }
 we can always write

\noi  
\begin{align}\label{decomp1}
X = \sum_{i=1}^\infty  X_i  k_i
\quad
\text{with  $X_i : = \langle X, k_i \rangle_K$}.
\end{align}

\begin{definition} \label{def1}
Let $K$ be a real separable Hilbert space. We say $X$ is a Gaussian random
variable on $K$ if
$\langle X, k \rangle_K$ is Gaussian for any $k\in K$.
We say $X$ is centered,
 if $\langle X, k \rangle_K$ is centered for any $k\in K$.
We say $X$ is non-degenerate,
if $\langle X, k \rangle_K$ is not a constant for any  nonzero $k\in
K$.\footnote{In other words, a constant is viewed
as a degenerate Gaussian random variable (with zero variance).
When the dimension of $K$ is finite, the non-degeneracy is equivalent to that
the finite-dimensional
Gaussian vector has a non-singular covariance matrix.}

\end{definition}

For $p\in[1,\infty)$,
we denote by $L^p(\Omega; K)$ the set of all $K$-valued
random variables $X$ satisfying
\begin{align*}
\| X \|_{L^p(\O;K)} : = \big(  \E\big[ \| X \|_K^p \big] \big)^{\frac1p} <
+\infty.
\end{align*}
The space $L^\infty(\O;K)$ consists of random variables $X$
with $\|X\|_K\in L^\infty(\O; \R)$.
And for $X\in L^1(\O;K)$, the expectation $\E[X]$,  interpreted  as
the Bochner integral
\begin{align*}
\E[X] := \int_\O X d\P \in K
\end{align*}
is well defined.
For $X\in L^2(\O;K)$, its covariance operator is
$$
S_Xu=\E[\langle X-\E[X],u\rangle_K(X-\E[X])].
$$
In the following identities, assume that $X$ is centered.
Then $S_X$ is a bounded linear operator on $K$ given by

\noi
\begin{align}\label{SX1}
u\in K \longmapsto S_X u := \E\big[   \langle X, u \rangle_K X \big].
\end{align}

\noi
Using the decomposition \eqref{decomp1}, we have\footnote{Note that
the assumption $X\in L^2(\O; K)$, together with the
property of Bochner integration, implies that
$\| S_X u\|_K \leq \E[ \|X\|_K^2] \|u\|_K$.
It is also evident that
$\E \sum_{i=1}^\infty X_i^2 <\infty$,
which yields  the finiteness of the expectations in \eqref{SX2}
and interchange of sum and expectation therein.   }

\noi
\begin{align}\label{SX2}
\langle S_X u, k_j\rangle_K
= \E\bigg( \sum_{i=1}^\infty X_i u_i X_j \bigg)
= \sum_{i=1}^\infty \E[ X_j X_i ] u_i ,
\end{align}

\noi
from which one can easily see that the covariance operator
is a generalization of the covariance matrix of a
{\it centered}
random vector in a finite-dimensional Euclidean space.
It is also clear from \eqref{SX1} that

\noi
\begin{align}\label{tr}
\begin{aligned}
{\rm tr}(S_X) : = \sum_{i=1}^\infty \langle S_X k_i, k_i \rangle_K
= \sum_{i=1}^\infty \E[ X_i^2 ] = \E[ \| X\|_K^2 ].
\end{aligned}
\end{align}
That is, $S_X$ is a trace-class operator on $K$.
Similarly, we have

\noi
\begin{align}\label{HS}
\begin{aligned}
\| S_X\|^2_{\rm HS} :& = \sum_{i=1}^\infty  \| S_X k_i \|^2_K
= \sum_{i=1}^\infty  \| \E[ X_i X ] \|_K^2 \\
&\leq  \sum_{i=1}^\infty  \big( \E \big[ | X_i| \cdot \|X\|_K \big] \big)^2
\leq  \big( \E[ \|X\|_K^2 ] \big)^2=\tr(S_X)^2.
\end{aligned}
\end{align}
That is, $S_X$ is also an Hilbert-Schmidt operator on $K$ with
the Hilbert-Schmidt norm $\| S_X\|_{\rm HS}$ bounded by the trace ${\rm
tr}(S_X)$.
Then, we have

\noi
\begin{align}\label{norm_bdd}
\| S_X \|_{\rm op} \leq \| S_X\|_{\rm HS} \leq {\rm tr}(S_X),
\end{align}

\noi
where $\| \cdot \|_{\rm op}$   denotes the usual operator norm
of an     operator on $K$.

In this paper, we focus on normal approximation of $K$-valued
random variables subordinated to an isonormal Gaussian
process $W =  \{ W(h): h\in\fH\}$, with $\fH$ another real separable
Hilbert space.
\begin{itemize}

\item[(i)]  We say $W$ is an isonormal
Gaussian process over $\fH$ if $W$ is a {\it real} centered Gaussian family
with covariance
structure
\begin{align*}
\E\big[ W(h_1) W(h_2) \big] = \langle h_1, h_2 \rangle_\fH
\end{align*}
for any $h_1, h_2\in\fH$.

\item[(ii)] To quantify the proximity of distributions of two
random variables $F,G\in L^1(\Omega;K)$,
we employ the $d_2$ metric defined by

\noi
\begin{align}\label{d2}
d_2(F, G) : = \sup_{\phi} \big|  \E[ \phi(F)] - \E[ \phi(G) ] \big|,
\end{align}
where the supremum runs over all functions $\phi: K\to \R$
that are twice Fr\'echet differentiable  with uniformly bounded first
and second derivatives:
\begin{align}\label{HSH1}
\max_{j=1,2} \sup_{v\in K}  \| \nb^j \phi(v) \|_{K^{\otimes j}} \leq 1
\end{align}
with $\nb^j$ denoting the $j$-th Fr\'echet derivative
and $K^{\otimes j}$ the   $j$-th tensor product   of $K$;
see Section 2.2.4 in \cite{BC20}.

\end{itemize}

The norm on $K^{\otimes2}$ in \eqref{d2}
is the Hilbert-Schmidt norm of the Hessian.
As shown in \cite[Theorem 1]{BBCP26}, convergence in this
distance does not in general imply weak convergence on $K$.
For example, the counterexample in the proof of that theorem
is given by
\begin{align*}
 X_n=\frac{1}{\sqrt n}\sum_{j=1}^n g_j k_j,
 \qquad g_j\overset{\rm iid}{\sim}\NN(0,1),
\end{align*}
where $\{k_j:j\geq1\}$ is an orthonormal basis
of an infinite-dimensional $K$.
The Gaussian interpolation argument in
\cite[Proposition 2]{BBCP26} gives
$d_2(X_n,0)\leq(2\sqrt n)^{-1}$,
whereas the strong law of large numbers yields
$\|X_n\|_K^2\to1$ almost surely.
Hence $X_n$ does not converge weakly to zero,
despite $\E [ \|X_n\|_K^2] =1$ for every $n$.
Throughout this paper, weak convergence on $K$
refers to its norm topology.

For $F,G\in L^2(\Omega;K)$, we put
\begin{align}\label{d2dag}
 d_2^\dagger(F,G)
 :=d_2(F,G)
 +\big|\E[\|F\|_K^2]-\E[\|G\|_K^2]\big|.
\end{align}
We write $\mathcal P_2(K)$ for the Borel probability measures
on $K$ with finite second moment and use the notation
$W_2(F,G)=W_2(\Law(F),\Law(G))$, where
\begin{align*}
 W_2(F,G)
 :=\inf\big\lbrace\big(\E[\|F'-G'\|_K^2]\big)^{1/2}:
       \Law(F')=\Law(F),\ \Law(G')=\Law(G)\big\rbrace.
\end{align*}
The infimum runs over all couplings of the two laws.
The following proposition records a consequence of the
standard projection argument and the classical
Wasserstein moment criterion
\cite[Definition 6.8 and Theorem 6.9]{Vil09}.
We include its proof for completeness.

\begin{proposition}\label{prop_d2mom}
Let $X_n,X\in L^2(\Omega;K)$.
If
\begin{align}\label{d2mom_cond}
 d_2(X_n,X)\to0
 \quad
 \text{and}
 \quad
 \E[\|X_n\|_K^2]\to\E[\|X\|_K^2],
\end{align}
then $X_n$ converges weakly to $X$ on $K$
and $W_2(X_n,X)\to0$.
Moreover, $d_2^\dagger$ is a metric on $\mathcal P_2(K)$
and metrizes $W_2$-convergence.
\end{proposition}

In this work, we use $W_2$-convergence to establish weak convergence,
without seeking quantitative bounds in $W_2$.
The weak-convergence conclusion of
Proposition~\ref{prop_d2mom}
also follows from Cremers and Kadelka
\cite[Theorem 2, condition (4.3)]{CK86}.
Their result gives weak convergence in $L^p$
from convergence of the pointwise
finite-dimensional distributions and
an integrated $p$th-moment condition.
Taking $p=2$ and counting measure on $\mathbb N$,
we identify $K$ with $\ell^2$
through an orthonormal basis.
The sine and cosine tests yield convergence
of the coordinate finite-dimensional distributions,
while Parseval's identity identifies
their moment condition with the second condition
in \eqref{d2mom_cond}.
We include a direct proof for completeness.

\begin{proof}[Proof of Proposition \ref{prop_d2mom}]
Let $P_m$ denote the orthogonal projection onto
$\operatorname{span}\{k_1,\ldots,k_m\}$.
If $K$ is finite-dimensional, take $P_m=\Id$
for $m\geq\dim K$.
For each $h\in K$, the functions
$x\mapsto\cos\jb{x,h}_K$ and
$x\mapsto\sin\jb{x,h}_K$, divided by
$1\vee\|h\|_K\vee\|h\|_K^2$, belong to the test class
in \eqref{d2}.
Thus the first condition in \eqref{d2mom_cond}
implies convergence of the characteristic functions
of $P_mX_n$ to those of $P_mX$
as $n\to+\infty$.
That is,
$P_mX_n$ converges weakly to $P_mX$ for every fixed $m$.
Therefore, by 
applying  this weak convergence to
$x\mapsto\min\{\|x\|_K^2,R\}$
and then letting $R\to\infty$,
we can have
\begin{align*}
 \E[\|P_mX\|_K^2]
 \leq \liminf_{n\to\infty}
 \E[\|P_mX_n\|_K^2].
\end{align*}
Together with the second condition in \eqref{d2mom_cond}
and the orthogonality of $P_m$, this gives
\begin{align}\label{d2mom_tail}
 \limsup_{n\to\infty}
 \E[\|(\Id-P_m)X_n\|_K^2]
 &= \E[\|X\|_K^2]
    -\liminf_{n\to\infty}\E[\|P_mX_n\|_K^2]
 \notag\\
 &\leq \E[\|(\Id-P_m)X\|_K^2].
\end{align}
Let $d_{\rm BL}$ be the bounded Lipschitz distance,
whose test functions have supremum norm
and Lipschitz constant at most one.
By the triangle inequality,
\begin{align*}
 d_{\rm BL}(X_n,X)
 &\leq \E[\|(\Id-P_m)X_n\|_K]
       +d_{\rm BL}(P_mX_n,P_mX)
  +\E[\|(\Id-P_m)X\|_K].
\end{align*}
For fixed $m$, the middle term tends to zero.
Using \eqref{d2mom_tail} and Cauchy-Schwarz, we obtain
\begin{align*}
 \limsup_{n\to\infty}d_{\rm BL}(X_n,X)
 \leq 2\big(\E[\|(\Id-P_m)X\|_K^2]\big)^{1/2}
 \xrightarrow{m\to\infty}0.
\end{align*}
Hence $X_n$ converges weakly to $X$.
The second-moment convergence and
\cite[Definition 6.8 and Theorem 6.9]{Vil09}
then imply $W_2(X_n,X)\to0$.

The symmetry and triangle inequality for $d_2^\dagger$
follow directly from \eqref{d2dag}.
The same sine and cosine tests show that
$d_2(F,G)=0$ implies equality of the characteristic
functionals of $F$ and $G$.
Since $K$ is separable, this yields
$\Law(F)=\Law(G)$.
As a result,  $d_2^\dagger$ is a metric.
We have already proved that convergence in $d_2^\dagger$
implies convergence in $W_2$.
Conversely, every test function in \eqref{d2}
is $1$-Lipschitz, and therefore $d_2(F,G)\leq W_2(F,G)$.
Convergence in $W_2$ also implies convergence
of the second moments, by the same Wasserstein criterion.
This proves the converse implication.
\end{proof}

\begin{remark}\label{rem:d2dag}\rm
For centered $F,G\in L^2(\Omega;K)$,
\eqref{tr} and \eqref{d2dag} yield
\begin{align}\label{d2dag_cov}
 d_2^\dagger(F,G)
 =d_2(F,G)+\big|{\rm tr}(S_F)-{\rm tr}(S_G)\big|.
\end{align}
In particular, $d_2^\dagger(F,G)=d_2(F,G)$
whenever $S_F=S_G$.
The metric $d_2^\dagger$ metrizes weak convergence
together with convergence of the second moments.

For centered Gaussian laws, the classical criterion
recalled in \cite[Proposition 1]{BBCP26}
states that weak convergence is equivalent to
operator-norm convergence of the covariance operators
and convergence of their traces.
A quantitative Gaussian comparison with an explicit
scalar trace correction appears in
\cite[Propositions 5.11--5.12]{FHMNP23}:
$W_2$ is bounded by the Hilbert-Schmidt distance between
the covariance square roots, which is in turn controlled
by the scalar trace difference and the Hilbert-Schmidt
distance between the covariances.
The latter bound also involves
$\min\{{\rm tr}(\sqrt{S_F}),{\rm tr}(\sqrt{S_G})\}$,
whose finiteness imposes an additional spectral
summability assumption on at least one covariance.

Other treatments employ trace-norm bounds.
The fourth-moment theorem in \cite{DZ25}
assumes trace-norm convergence of the covariance operators;
this framework is used in the  
Breuer--Major theorem of \cite{DZ24}.
Bourguin and Campese
\cite[Theorem 7.3 and Remark 7.4]{BC25}
control an operator-Hessian distance by the expected
trace norm of the Malliavin--Stein discrepancy.
These operator bounds differ from the scalar correction
in \eqref{d2dag_cov}.

\end{remark}

Now let us introduce Hilbert-valued Malliavin calculus and then
state the abstract bound of Bourguin and Campese \cite{BC20},
from which we further derive the (improved)
second-order Gaussian Poincar\'e inequalities.

\subsection{Chaos expansions and Malliavin operators}  \label{SEC22}

It is well known that the $L^2(\P)$ space of real-valued
random variables that are measurable with respect to
$\s\{ W\}$ admits the Wiener-It\^o chaos decomposition:
for any $Y\in L^2(\O, \s\{W\}, \P ; \R)$,
there exist unique kernels
$g_n\in\fH^{\odot n}$ for $n\in\N_{\geq 1}$
such that

\noi
\begin{align}\label{chaos1}
Y = \E[ Y] + \sum_{n\geq 1} I_n( g_n),
\end{align}

\noi
where $I_n$ denotes the $n$-th multiple integral
and $\fH^{\odot n}\subset \fH^{\otimes n}$ denotes
the $n$-th symmetric tensor product of $\fH$;
see \cite[Chapter 1]{Nua06}
and \cite[Chapter 2]{NP12}
for definitions and basic properties.
To develop the Malliavin calculus
in the Hilbert-valued setting,
it suffices to apply the real theory
along  each direction $k_i$, in view
of the decomposition \eqref{decomp1}.

\bigskip

For any $p\in[1,\infty)$, we put $L^p_\o K :=  L^p(\O, \s\{W\}, \P ; K)$.

\begin{lemma}\label{WID}
{\rm (i)} \textup{(Wiener chaos)}  Fix $n\in\N_{\geq 1}$.
Given  $f\in \fH^{\odot n}\otimes K$,

\noi
\begin{align} \label{IN}
I_n(f) : = \sum_{i\geq 1} I_n(f_i) k_i \quad
\text{with}\quad f_i = \langle f, k_i\rangle_K
\end{align}
is well defined and also called the $n$-th multiple integral
of $f$.\footnote{We use the same notation for the real setting
and $K$-valued setting, which shall not cause any ambiguity
for readers. }
For a {\rm(}deterministic{\rm)} vector $v\in K$, we put
$I_0(v) := v$.
And, the following orthogonality relation holds:
\begin{align}\label{OR}
\langle I_n(f) , I_m(g) \rangle_{L^2_\o K} =
\ind_{\{ n=m\}} n! \jb{f, g}_{  \fH^{\otimes n}\otimes K  }
\end{align}
for any $m,n\in\N_{\geq 0}$ and for any $f\in \fH^{\odot n}\otimes K$
and $g\in \fH^{\odot m}\otimes K$.

\medskip

\noi
{\rm (ii)} \textup{(Chaos decomposition)}
For any $F\in L^2_\o K$, there exist
unique kernels
$f_n\in\fH^{\odot n}\otimes K$ for $n\in\N_{\geq 1}$
such that

\noi
\begin{align}\label{chaos2}
F = \E[ F] +   \sum_{n\geq 1} I_n(f_n) ,
\end{align}

\noi
where the above series is convergence in $L^2_\o K$.
In particular, we have

\noi
\begin{align}
 \| F - \E[F] \|^2_{L^2_\o K} =  \sum_{n\geq 1} n!   \| f_n \|^2_{\fH^{\otimes
n} \otimes K}.
 \label{VarF}
\end{align}

\end{lemma}

\begin{proof}
(i) Suppose $f\in    \fH^{\odot n}\otimes K$.
Then, for any $i\in\N_{\geq 1}$,
$f_i=\jb{f, k_i}_K\in\fH^{\odot n}$
with
\begin{align*}
\| f\|^2_{\fH^{\otimes n}\otimes K}
= \sum_{i\geq 1} \| f_i\|^2_{\fH^{\otimes n}}.
\end{align*}
Thus,
$I_n(f_i) $ is well defined with
$\E\big[ |I_n(f_i) |^2 \big] = n! \| f_i \|^2_{\fH^{\otimes n}}$
and
\begin{align*}
\bigg\| \sum_{i\geq 1} I_n(f_i) k_i \bigg\|^2_{L^2_\o K}
= \sum_{i\geq 1} \E\big[ |I_n(f_i) |^2 \big] = n! \| f\|^2_{\fH^{\otimes
n}\otimes K}.
\end{align*}
That is, the definition \eqref{IN} makes sense and the orthogonality
relation holds for $f= g \in  \fH^{\odot n}\otimes K$.
Similarly, for any $f\in \fH^{\odot n}\otimes K$
and $g\in \fH^{\odot m}\otimes K$,
we deduce from Parseval's identity and  orthogonality
relation in the real setting (see, e.g.,
 \cite[Proposition 2.7.5]{NP12})
that

\noi
\begin{align*}
\langle I_n(f) , I_m(g) \rangle_{L^2_\o K}
&=  \sum_{i\geq 1} \E\big[ I_n(f_i)  I_m(g_i) \big] \\
&=   \ind_{\{ n=m\}} n! \sum_{i\geq 1} \jb{f_i, g_i}_{  \fH^{\otimes n} }
= \ind_{\{ n=m\}} n! \jb{f, g}_{  \fH^{\otimes n}\otimes K  }.
\end{align*}
That is, the orthogonality relation \eqref{OR} is established.

Next, we  prove the chaos expansion.

\medskip
\noi
(ii) Let $F\in L^2_\o K$. We can first write $F  = \sum_{i\geq 1} F_i k_i$
with
$F_i = \jb{F, k_i}_K\in L^2_\o \R$.
By the chaos expansion \eqref{chaos1} in the real setting,
we can find unique kernels $f_{n, i} \in \fH^{\odot n}$, $n\geq 1$,
such that
\begin{align*}
F_i =  \E[ F_i ]  + \sum_{n\geq 1} I_n( f_{n, i} ).
\end{align*}
In particular, we have
\begin{align} \label{T0}
\sum_{i, n\geq 1} \| f_{n, i} \|^2_{\fH^{\otimes n}} <\infty.
\end{align}

\noi
On one hand, it is clear that $\sum_{i=1}^M F_i k_i$ converges in $L^2_\o K$
to $F$
as $M\to+\infty$. On the other hand, we have

\noi
\begin{align}
\sum_{i=1}^M F_i k_i
&= \sum_{i=1}^M   \bigg(   \E[ F_i ]  + \sum_{n\geq 1} I_n( f_{n, i} ) \bigg)
k_i  \notag \\
&=  \bigg(  \sum_{i=1}^M  \E[ F_i ]  k_i \bigg)
 +    \bigg(   \sum_{n\geq 1}  \sum_{i=1}^M  I_n( f_{n, i} )  k_i  \bigg),
\label{T1}
\end{align}

\noi
where the first term  in \eqref{T1}  converges to $\E[F]$, and
$ \sum_{i=1}^M  I_n( f_{n, i} )  k_i  = I_n\big(   \sum_{i=1}^M  f_{n, i}
\otimes   k_i \big)  $
converges in $L^2_\o K$ to $ I_n (   f_n\big) $ with
\begin{align*}
f_n =  \sum_{i=1}^\infty  f_{n, i} \otimes   k_i   \in\fH^{\odot n}\otimes K.
\end{align*}
Note that the membership of $f_n$ in $\fH^{\odot n}\otimes K$ follows from
\eqref{T0}.
Therefore, we deduce from the orthogonality relation in (i) that
the second term in \eqref{T1} converges in
$L^2_\o K$ to $\sum_{n\geq 1} I_n (   f_n\big) $.
Thus the proof of (ii) (and hence that of Lemma \ref{WID}) is completed.
\qedhere

\end{proof}

\subsubsection{Contractions.}
Let us recall its definition  in the real setting (see, e.g., \cite[Appendix
B]{NP12}).
Let $f\in \fH^{\otimes m}$ and $g\in  \fH^{\otimes n}$ with $m, n\in\N_{\geq
1}$.
$f\otimes_0 g = f\otimes g$ is the usual tensor product of $f$ and $g$.
For $r=1, ... , \min\{m, n\}$, the $r$-contraction between $f$ and $g$ is
an element in $\fH^{\otimes m+n-2r}$ defined by

\noi
\begin{align}\label{contra1}
\begin{aligned}
&f \otimes_r g
= \sum_{\substack{i_1, ... , i_{m-r} \\ j_1, ... , j_{n-r}  \\ \ell_1, ... ,
\ell_r  }}
\langle f, h_{i_1}\otimes \cdots \otimes h_{i_{m-r}}\otimes h_{\ell_1}
\otimes \cdots \otimes h_{\ell_r} \rangle_{\fH^{\otimes m}} \\
&  \quad \times
\langle g, h_{j_1}\otimes \cdots \otimes h_{j_{n-r}}\otimes h_{\ell_1}
\otimes \cdots \otimes h_{\ell_r} \rangle_{\fH^{\otimes n}}
h_{i_1}\otimes \cdots \otimes h_{i_{m-r}}\otimes
h_{j_1}\otimes \cdots \otimes h_{j_{n-r}}.
\end{aligned}
\end{align}

\noi
Now consider $f\in \fH^{\otimes m}\otimes K$ and $g\in  \fH^{\otimes n}
\otimes K$,
we can define
\begin{align*}
f\otimes_r g = \sum_{i,j\geq 1} \big( \jb{f, k_i}_K \otimes_r  \jb{g, k_j}_K
\big)\otimes k_i \otimes k_j
\end{align*}
for $r=0,1, ... , \min\{m, n\}$,
where $\jb{f, k_i}_K \otimes_r  \jb{g, k_j}_K $ is defined as in
\eqref{contra1}.

 \smallskip

With the above chaos expansion, we introduce the
Malliavin operators in the spirit of \cite{NV90}.

\subsubsection{Malliavin derivative.}
Recall from \cite[Chapter 2]{NP12} that
the space $\mathbb D^{1,2}$ of real-valued,  Malliavin differentiable, and
square-integrable random variables can be defined by

\noi
\begin{align}
\mathbb D^{1,2}
= \{  Y\in L^2_\o\R \,\, \text{as in \eqref{chaos1}}:
\,\, \sum_{n\geq 1}n n! \| g_n\|^2_{\fH^{\otimes n}} < +\infty    \}.
\label{DR}
\end{align}
For  $Y\in \mathbb D^{1,2}$ as in \eqref{chaos1},
we have
$
DY = \sum_{n\geq 1} n I_{n-1} (   g_n\big)  \in L^2(\O; \fH).
$
Similarly, we can define
\noi
 \begin{align}
\mathbb D^{1,2}(K)
:= \{  F\in L^2_\o K\,\, \text{as in \eqref{chaos2}}:
\,\, \sum_{n\geq 1} n n! \| f_n\|^2_{\fH^{\otimes n} \otimes K} < +\infty
\},
\label{DK}
\end{align}
and for such a random variable $F\in \mathbb D^{1,2}(K) $, we define

\noi
\begin{align}\label{exp_DF}
DF := \sum_{n\geq 1} n I_{n-1} (   f_n\big)  \in L^2(\O; \fH\otimes K),
\end{align}

\noi
which satisfies
$\E\big[ \| DF \|^2_{\fH\otimes K} \big]
= \sum_{n\geq 1} n n! \| f_n\|^2_{\fH^{\otimes n} \otimes K} $.
A comparison of this expression with \eqref{VarF} leads to the
Gaussian Poincar\'e inequality:
it holds for $F\in \mathbb D^{1, 2}(K)$ that

\noi
\begin{align}\label{GPI}
 \big\| F - \E[F] \big\|^2_{L^2_\o K} \leq \E\big[ \| DF \|^2_{\fH\otimes K}
\big]
\end{align}
with equality when and only when $F$ lives in the direct sum of $K$ and first
chaos.
Note also that $DF$ coincides with

\noi
\begin{align} \label{C1}
\sum_{i=1}^\infty \big( D \langle F, k_i \rangle_K \big) \otimes  k_i
= \sum_{i=1}^\infty  \sum_{n= 1}^\infty n I_{n-1} (  \langle f_n, k_i
\rangle_K\big)  \otimes  k_i.
\end{align}

\noi
Alternatively, assuming
\begin{align}\label{OBH}
\text{$\{h_j: j\geq1\}$ is an orthonormal basis of $\fH$, }
\end{align}
we can write

\noi
\begin{align}\label{C2a}
DF = \sum_{j=1}^\infty   \jb{DF, h_j}_\fH  \otimes h_j
= \sum_{j=1}^\infty    \bigg( \sum_{n=1}^\infty n I_{n-1} \big(  \jb{f_n, h_j
}_\fH \big)  \bigg) \otimes h_j ,
\end{align}

\noi
where $\jb{f_n, h_j }_\fH \in\fH^{\odot (n-1)}\otimes K$ and
$I_{n-1} \big(  \jb{f_n, h_j }_\fH \big) $ is defined according to Lemma
\ref{WID}.
It  is easy to see from \eqref{C1}
 that for $F = f( W(\phi_1), ... , W(\phi_m) ) v$ with
$v\in K$,  $f:\R^m \to \R$ bounded smooth,
and
$\phi_j \in \fH$,
we have
\begin{align}\label{c_chain}
DF
= \big( D f( W(\phi_1), ... , W(\phi_m) ) \big) \otimes v
= \sum_{i=1}^m \partial_{i} f( W(\phi_1), ... , W(\phi_m) )  \phi_i \otimes v,
\end{align}

\noi
which is exactly stated in \cite[(4.1)]{BC20}. We can also define iterated
Malliavin derivatives.
For our purpose, we only define
\begin{align*}
\mathbb{D}^{2,2}(K) : =
\{  F\in L^2_\o K\,\, \text{as in \eqref{chaos2}}:
\,\, \sum_{n\geq 1} n^2 n! \| f_n\|^2_{\fH^{\otimes n} \otimes K} < +\infty
\}
\end{align*}
and for such $F\in \mathbb{D}^{2,2}(K) $,
\begin{align*}
D^2F =  \sum_{n\geq 2} n(n-1) I_{n-2} (   f_n\big)  \in L^2(\O; \fH^{\odot
2}\otimes K).
\end{align*}
We also use $\mathbb D^{2,4}(K)$, the closure of smooth
cylindrical $K$-valued random variables for the norm
\begin{align*}
 \|F\|_{2,4;K}
 :=\bigg(\E[\|F\|_K^4]
 +\E[\|DF\|_{\fH\otimes K}^4]
 +\E[\|D^2F\|_{\fH^{\otimes2}\otimes K}^4]
 \bigg)^{1/4}.
\end{align*}
Here a smooth cylindrical random variable is a finite sum
of terms $f(W(h_1),\ldots,W(h_m))k$, with
$f\in C_b^\infty(\R^m)$, $h_i\in\fH$, and $k\in K$.
This definition agrees with the preceding derivatives
and gives $\mathbb D^{2,4}(K)\subset\mathbb D^{2,2}(K)$.

\subsubsection{Divergence operator.} As in the real case,  the divergence
operator $\dl$ is
defined as the adjoint operator of $D$
and is characterized by the following duality relation:

\noi
\begin{align}
\E\big[ \langle DF, V \rangle_{\fH\otimes K} \big] = \E[ \jb{F , \dl(V)}_K ]
\label{dualR}
\end{align}

\noi
for any $F\in \mathbb{D}^{1,2}(K)$.
In view of Riesz's representation theorem, we let $\dom(\dl)$ be the set of
$V\in L^2_\o(\fH\otimes K)$ such that there is some finite constant $C=C(V) >
0$ such that
\begin{align*}
\big| \E\big[ \langle DF, V \rangle_{\fH\otimes K} \big]  \big| \leq C \| F
\|_{L^2_\o K}
\end{align*}
for any $F\in\mathbb{D}^{1,2}(K)$.
Then, the duality relation \eqref{dualR} holds for any $(F, V)\in
\mathbb{D}^{1,2}(K)\times \dom(\dl)$.

Suppose $V\in L^2_\o(\fH\otimes K)$. Then, for every $h_j\in \fH$ as in
\eqref{OBH},
$\jb{V, h_j}_\fH \in L^2_\o K$.
 Then, by chaos decomposition (Lemma \ref{WID}-(ii))
we can write
\begin{align}
\jb{V, h_j}_\fH = \E \big[ \jb{V, h_j}_\fH \big] + \sum_{n=1}^\infty I_n
(g_{n,j}  ) ,
\label{CD7}
\end{align}

\noi
where $g_{n,j}   \in  \fH^{\odot n} \otimes K$ for each $j\in\N_{\geq 1}$
 and we write $g_{0, j}  = \E  [ \jb{V, h_j}_\fH  ].$
Note that $V\in L^2_\o(\fH\otimes K)$ forces

\noi
\begin{align}
g_n: =  \sum_{j\geq 1} g_{n,j} \otimes h_j   \in \fH^{\otimes (n+1)} \otimes K
\quad \text{for every $n\in\N_{\geq 1}$.}
\label{GM}
\end{align}

\noi
Note that $g_n$ may not belong to $\fH^{\odot (n+1)} \otimes K$.  We denote by
$\wt{g}$ the canonical symmetrization of $g \in \fH^{\otimes n} \otimes K  $
in the $n$ coordinates of $ \fH^{\otimes n}$:

\noi
\begin{align}\label{SYM}
\wt{g}  = \frac{1}{n!} \sum_{\s\in\mathfrak{S}_n} \sum_{i_1, ... , i_n\geq 1}
h_{i_{\s(1)}}  \otimes \cdots \otimes h_{i_{\s(n)}}
\langle g, h_{i_1}\otimes \cdots \otimes h_{i_n} \rangle_{\fH^{\otimes n}}  ,
\end{align}

\noi
where $\mathfrak{S}_n$ denotes the set of permutations on $\{1, ... , n\}$.
It is clear that $\wt{g}\in \fH^{\odot n} \otimes K$.
For non-symmetric kernel $g$, we define

\noi
\begin{align}
 \text{$I_n(g) = I_n(\wt{g})$
 for $g \in \fH^{\otimes n} \otimes K$.}
\label{Nsym}
\end{align}

Assume first that there are finitely many chaoses in the above series
\eqref{CD7}:

\noi
\begin{align}
\text{$g_{n, j}  = 0$ for $n >  M$ and $j\geq 1$.}
\label{CD7b}
\end{align}

\noi
Then, for $F\in\mathbb{D}^{1,2}(K)$ having the form \eqref{chaos2},
we  deduce from Parseval's identity on $\fH$,  \eqref{C2a}, \eqref{CD7},
Fubini's theorem,
 orthogonality relation \eqref{OR}, and \eqref{GM} that

\noi
\begin{align}
\begin{aligned}
\E \big[ \langle DF, V \rangle_{\fH\otimes K} \big]
&= \E \sum_{j\geq 1}  \Big\langle  \jb{DF, h_j}_\fH,  \jb{V, h_j}_\fH
\Big\rangle_K \\
&=   \E \sum_{j\geq 1}  \Big\langle
        \sum_{n=1}^\infty n I_{n-1} \big(  \jb{f_n, h_j }_\fH \big)   ,
       \sum_{m=0}^M I_m (g_{m,j}  )  \Big\rangle_K \\
&= \sum_{j\geq 1} \sum_{m=0}^M (m+1)!  \big\langle \jb{f_{m+1}, h_j }_\fH,
                    g_{m,j } \big\rangle_{\fH^{\otimes m} \otimes K} \\
& =    \sum_{m=0}^M (m+1)!  \big\langle  f_{m+1},
                    g_{m } \big\rangle_{\fH^{\otimes m+1} \otimes K}  \\
&   =  \sum_{m=0}^M (m+1)!  \big\langle  f_{m+1},
                   \wt g_{m } \big\rangle_{\fH^{\otimes m+1} \otimes K},
\end{aligned}
\label{CD8}
\end{align}

\noi
 which, together with Cauchy-Schwarz's inequality,
implies that

\noi
\begin{align}
\begin{aligned}
\big| \E \big[ \langle DF, V \rangle_{\fH\otimes K} \big] \big|
&\leq  \bigg(\sum_{n=1}^{M+1} n! \| f_n\|^2_{\fH^{\otimes n} \otimes K}
\bigg)^{\frac12}
\bigg(\sum_{m=0}^M (m+1)! \|  \wt{g}_m \|^2_{\fH^{\otimes {m+1}} \otimes K}
\bigg)^{\frac12} \\
&\leq  \| F\|_{L^2_\o K}
\bigg(\sum_{m=0}^M (m+1)! \|  \wt{g}_m \|^2_{\fH^{\otimes {m+1}} \otimes K}
\bigg)^{\frac12} .
\end{aligned}
\label{CD9}
\end{align}

\noi
In particular, we proved that for $V\in L^2_\o(\fH\otimes K)$ satisfying
\eqref{CD7b},
$V$ belongs to $\dom(\dl)$;\footnote{This
also tells us that $\dom(\dl)$ is dense in $L^2_\o(\fH\otimes K)$.}
and in this case, we deduce again from \eqref{CD8}
and \eqref{OR} that

\noi
\begin{align}
\begin{aligned}
\E \big[ \langle DF, V \rangle_{\fH\otimes K} \big]
& =  \sum_{m=0}^M  \E\big[ \jb{ I_{m+1}( f_{m+1}),  I_{m+1}(  \wt{g}_m ) }_K
\big] \\
&= \E \bigg[ \Big\langle  F,   \sum_{m=0}^M I_{m+1}(\wt{g}_m   ) \Big\rangle_K
\bigg]
\end{aligned}
\label{CD9b}
\end{align}
for any $F\in\mathbb{D}^{1,2}(K)$, and thus,
\begin{align}
\dl(V) =  \sum_{m=0}^\infty I_{m+1}(\wt{g}_m   ) .
\label{CD9c}
\end{align}

One can easily generalize this particular case of \eqref{CD7b}
to the following result, whose proof
is omitted; see also \cite[Lemma 2.4]{BZ24}
and  \cite[Lemma 2.12]{BZ25} for similar arguments.

\begin{lemma}\label{lem:dl}
Suppose $V\in L^2_\o(\fH\otimes K)$
has the expression \eqref{CD7} with   \eqref{GM}.
Then,
  $V\in\dom(\dl)$ and $\dl(V)$ is given as in \eqref{CD9c}
  if and only if

  \noi
\begin{align}\label{cond:dl}
\sum_{m=0}^\infty (m+1)! \| \wt{g}_m\|^2_{\fH^{\otimes (m+1)} \otimes K}
<\infty.
\end{align}

\noi
In particular, $\E[ \dl(V) ]$ is the zero vector in $K$ when \eqref{cond:dl}
holds.

\end{lemma}

As a consequence, for a deterministic function $\phi\in \fH \otimes K$,
we have
\begin{align}\label{EXT1}
 \dl(\phi) = I_1(\phi).
\end{align}

\subsubsection{Ornstein-Uhlenbeck semigroup and generator.}  The
Ornstein-Uhlenbeck  (OU)
semigroup $\{P_t: t\in\R_+\}$ on $L^2_\o K$ can be defined as follows:
for $F$ as in \eqref{chaos2},
\begin{align*}
P_t F = \E[ F] + \sum_{n\geq 1} e^{-nt } I_n(f_n).
\end{align*}
It is evident  that $\| P_t F \|_{L^2_\o K} \leq \|  F \|_{L^2_\o K}$
for any $F\in L^2_\o K$; see also \eqref{OU1}. The associated OU generator $L$
can be defined by

\noi
\begin{align*}
L F =   \sum_{n\geq 1} -n  I_n(f_n)
\end{align*}
for $F\in\dom(L) \equiv \mathbb D^{2,2}(K)$. In other words,
$F\in \dom(L)$ if and only if $F\in L^2_\o K$ verifies that
$t^{-1} (P_t F  - F)$ is convergent in $L^2_\o K$ as $t\downarrow 0$,
with the limit equal to $LF$. The following result also follows easily
from the basic definitions of $D, \dl$, and $L$.

\begin{lemma} \label{LdlD}
The identity $L = -\dl D$ holds:
 $F\in \dom(L)$ if and only if $F\in\mathbb{D}^{1,2}(K)$ and $DF\in\dom(\dl)$.
 In each case, we have $LF  = -\dl DF$.

\end{lemma}

\begin{proof}
Recall from \eqref{chaos2} that $F\in L^2_\o K$ admits the following chaos
expansion:
\begin{align*}
F = \E[F] +  \sum_{n\geq 1} I_n(f_n),
\end{align*}
where $f_n\in\fH^{\odot n} \otimes K$.

First, assume $F\in \dom(L)$. That is, we have
$
\sum_{n\geq 1} n^2 n! \| f_n \|^2_{ \fH^{\otimes n} \otimes K } <+\infty
$.
It follows that $F\in \mathbb{D}^{1,2}(K)$ with

\noi
\begin{align}\label{DF2}
DF =  \sum_{n\geq 1}n  I_{n-1}(f_n) \in L^2(\O; \fH \otimes K).
\end{align}

\noi
Put $V  = DF$. Then, it is clear that $V$ satisfies the assumptions in Lemma
\ref{lem:dl}:
\begin{itemize}
\item[(i)]
$\jb{V, h_j}_\fH = \sum_{n\geq 1}   I_{n-1}( g_{n,j} ) $ with
$g_{n,j} = n \jb{f_n, h_j}_\fH \in \fH^{\odot {n-1}} \otimes K$.

\item[(ii)] $g_n = \sum_{j\geq 1}  g_{n,j} \otimes h_j$ coincides with $n
f_n\in \fH^{\odot n}\otimes K$.

\item[(iii)] The condition \eqref{cond:dl} is clearly satisfied due to
the finiteness of $\sum_{n\geq 1} n^2 n! \| f_n \|^2_{ \fH^{\otimes n} \otimes
K }$.
\end{itemize}

\noi
Therefore, $DF\in \dom(\dl)$ with
\begin{align*}
\dl (DF) =  \sum_{n\geq 1}n  I_n(f_n) = - LF.
\end{align*}

Next, we assume that $F\in\mathbb{D}^{1,2}(K)$ and $DF\in\dom(\dl)$.
Then, $DF$ is given as in \eqref{DF2} such that

\noi
\begin{align}
V_M:= \sum_{n= 1}^M n  I_{n-1}(f_n)
 \quad
 \text{converges in $L^2(\O; \fH\otimes K)$ to $DF$ as $M\to+\infty$.}
\label{DF2b}
\end{align}
It is clear that $V_M\in\dom(\dl)$ with

\noi
\begin{align}
\dl(V_M) = \sum_{n= 1}^M n  I_{n}(f_n).
\label{DF2c}
\end{align}

\noi
 It follows from \eqref{DF2c}, duality relation \eqref{dualR},
 and  \eqref{DF2b}
that for any $G\in\mathbb{D}^{1,2}(K)$,

\noi
\begin{align}
\begin{aligned}
\E\Big[ \big\langle G,    \sum_{n= 1}^M n  I_{n}(f_n)    \big\rangle_K \Big]
&= \E\big[ \jb{G, \dl(V_M) }_K \big]  \\
&= \E\big[ \jb{DG,  V_M }_{\fH\otimes K} \big] \\
& \xrightarrow{M\to+\infty}
 \E\big[ \jb{DG,  DF }_{\fH\otimes K} \big] =  \E\big[ \jb{G,  \dl (DF) }_{K}
\big].
\end{aligned}
\label{DF2d}
\end{align}
To obtain a uniform bound, take $G=\dl(V_M)$.
This random variable belongs to a finite sum of chaoses,
so $G\in\mathbb D^{1,2}(K)$.
Then, we deduce from duality and orthogonality that

\noi
\begin{align*}
 \|\dl(V_M)\|_{L^2_\o K}^2
 &=\E[\langle D\dl(V_M),V_M\rangle_{\fH\otimes K}]\\
 &=\E[\langle D\dl(V_M),DF\rangle_{\fH\otimes K}]\\
 &=\E[\langle\dl(V_M),\dl(DF)\rangle_K]\\
 &\leq\|\dl(V_M)\|_{L^2_\o K}
       \|\dl(DF)\|_{L^2_\o K}.
\end{align*}
Consequently,
\begin{align*}
 \sum_{n=1}^\infty n^2n!\|f_n\|_{\fH^{\otimes n}\otimes K}^2
 =\sup_M\|\dl(V_M)\|_{L^2_\o K}^2
 \leq\|\dl(DF)\|_{L^2_\o K}^2<\infty.
\end{align*}
That is, $F\in \dom(L)$ and $\dl(V_M)= \sum_{n= 1}^M n  I_{n}(f_n) $ converges
to $\dl(DF)$
in  $L^2_\o K$. Hence, $LF = - \dl DF$, and the proof is completed. \qedhere

\end{proof}

\subsubsection{Mehler's formula.}
Next, we present a very useful formula that provides a representation of the
OU semigroup.
Let $W'$ be an independent copy of $W$, and define

\noi
\begin{align}\label{WT}
W^t = e^{-t} W + \sqrt{1-e^{-2t}}W',
\end{align}

\noi
which has the same law as $W$ for every $t\in\R_+$.
For any $F\in L^1_\o K$, one can represent it as a functional
of the isnormal process $W$ as follows.

\noi
\begin{itemize}

\item[(i)] For each $i\in\N_{\geq 1}$, $F_i = \jb{F, k_i}_K \in L^1_\o \R$.
Then, there is some measurable mapping $\Psi_i: \R^\fH \to \R$
(that is determined $\P\circ W^{-1}$ almost surely) such that
$F_i  = \Psi_i(W)$;
see, for example, \cite[Section 2.8.1]{NP12}.
 Therefore,
we can write

\noi
\begin{align}\label{rep1}
F = \sum_{i\geq 1} \Psi_i(W) k_i .
\end{align}

\item[(ii)] Alternatively, there is some measurable mapping $\Psi: \R^\fH \to
K$
(that is determined $\P\circ W^{-1}$ almost surely) such that

\noi
\begin{align}\label{rep2}
\text{$F  = \Psi(W)$ and $\jb{\Psi(W), k_i}_K = \Psi_i(W)$.}
\end{align}

\end{itemize}
Then, analogous to \cite[Theorem 2.8.2]{NP12}, we can derive
the following Mehler's formula:\footnote{One can understand the conditional
expectation of a $K$-valued random variable as in the second
equality in \eqref{Mehler}.
Or in an abstract way, one can see the conditional expectation
$\E[  \, \cdot \,  | \s\{W\} ]$
as a norm-one projection operator
from $L^p(\O, \cF, \P; K)$ onto $L^p(\O, \s\{W\}, \P; K)$
viewed as a subspace of $L^p(\O, \cF, \P; K)$,
$p\in[1,\infty]$,
such that $\E[  Y X  | \s\{W\} ] = Y \E[  X  | \s\{W\} ] $
for any $Y\in L^\infty(\O, \s\{W\}, \P;\R)$
and $X\in L^p(\O, \cF, \P; K)$;
see \cite[Section 1.2]{P16}.
}

\noi
\begin{align}
\begin{aligned}
P_t F
&= \E\big[ \Psi(W^t) | \s\{W\} \big] \\
&= \sum_{i\geq 1} \E\big[ \Psi_i(W^t) | \s\{W\} \big] k_i
\end{aligned}
\label{Mehler}
\end{align}

\noi
for $F = \Psi(W) \in L^1_\o K$ as in \eqref{rep1}-\eqref{rep2},
where $W^t$ is introduced in \eqref{WT}.

As a consequence of the Mehler's formula \eqref{Mehler},
we have
\begin{align}
\| P_t F \|_{L^p_\o K} \leq \|  F \|_{L^p_\o K}
\label{OU1}
\end{align}
for any $F\in L^p_\o K$ and $p\in [1, \infty]$.
Indeed, using Mehler's formula and properties of
conditional expectation (see, e.g.,
 \cite[Proposition 1.12]{P16}),
we have

\noi
\begin{align*}
\| P_t F \|_{L^p_\o K}
&= \big\|  \E\big[ \Psi(W^t) | \s\{W\} \big]  \big\|_{L^p(\O, \cF, \P; K)} \\
&\leq  \big\|  \E (   \|  \Psi(W^t) \|_K  | \s\{W\} )  \big\|_{  L^p(\O, \cF,
\P; \R)  }\\
&\leq \big\|    \Psi(W^t) \big\|_{L^p(\O, \cF, \P; K)}  = \|  F \|_{L^p_\o K}.
\end{align*}

\subsubsection{Pseudo-inverse of OU generator.} For  $F\in L^2_\o K$ as in
\eqref{chaos2},
\begin{align*}
L^{-1} F =  \sum_{n\geq 1}  -\frac{1}{n}I_n(f_n).
\end{align*}
Then,  $L^{-1}F\in \dom (L)$ with $LL^{-1}F = F - \E[ F]$.
Similarly, for $G\in\dom(L)$, we have   $L^{-1}LG = G - \E[G]$.
The operator $L^{-1}$ is called the pseudo-inverse of OU generator $L$.
It connects with the OU semigroup naturally:
for $F\in L^2_\o K$ with $\E[F] = 0$, we have

\noi
\begin{align}\label{rep3}
- L^{-1}F = \int_0^\infty P_t F dt.
\end{align}

\noi
To verify \eqref{rep3}, we begin with $F$ as in \eqref{chaos2} and $\E[F]=0$.
Then,
\begin{align*}
\text{RHS of \eqref{rep3}}
&=  \int_0^\infty  \sum_{n\geq 1} e^{-nt} I_n(f_n) dt \\
&=   \sum_{n\geq 1} \bigg( \int_0^\infty e^{-nt}  dt \bigg)I_n(f_n)= -
L^{-1}F,
\end{align*}
where the interexchange of the sum and integration follows from Fubini's
theorem
and the following bound:

\noi
\begin{align*}
&\quad \int_0^\infty  \sum_{n\geq 1} e^{-nt}  \|  I_n(f_n)  \|_K  \, dt
=   \sum_{n\geq 1}  \frac{1}{n}  \|  I_n(f_n)\|_K   \\
&\leq
 \bigg(  \sum_{n\geq 1}  \frac{1}{n^2}\bigg)^{\frac12}
  \Bigg[ \sum_{n\geq 1}    \|  I_n(f_n)\|^2_K     \Bigg]^{\frac12}
 = \frac{\pi}{\sqrt{6}} \Bigg[ \sum_{n\geq 1}    \|  I_n(f_n)\|^2_K
\Bigg]^{\frac12} \in L^2_\o \R.
\end{align*}
By similar arguments, we can get

\noi
\begin{align}\label{rep4}
-D L^{-1}F = \int_0^\infty D P_t F dt = \int_0^\infty e^{-t} P_t DF dt.
\end{align}
for any $F\in\mathbb{D}^{1,2}(K)$. Moreover, for any $p\in[1,\infty)$,
we can deduce from \eqref{rep4}, Minkowski's inequality, and \eqref{OU1}
that

\noi
\begin{align}  \label{rep4b}
\begin{aligned}
\| D L^{-1}F  \|_{L^p(\O; \fH\otimes  K)}
&\leq  \int_0^\infty e^{-t}  \| P_t DF  \|_{L^p(\O; \fH\otimes  K)}    dt \\
&\leq \| DF \|_{L^p(\O; \fH\otimes  K)}.
\end{aligned}
\end{align}

\subsection{Malliavin-Stein bounds} \label{SEC2_3}

Recall the definition \eqref{d2} of $d_2$.
Bourguin and Campese \cite[Theorem 3.2]{BC20}
bound this distance in terms of
$\Gamma(F,-L^{-1}F)$.
For the Wiener structure, their identity
\cite[(2.12)]{BC20} gives
\begin{align*}
 \Gamma(F,-L^{-1}F)
 =\frac12(M+M^*),
 \qquad
 M=\jb{DF,-DL^{-1}F}_\fH,
\end{align*}
where $M^*$ denotes the adjoint of $M$ on $K$.
Since $S_Z$ is self-adjoint,
$\| \frac{1}{2}(M+M^*)-S_Z\|_{\rm HS}\leq\|M-S_Z\|_{\rm HS}$.
Thus their estimate yields the following bound when $Z$
is non-degenerate. The proof below also covers degenerate $Z$.

\begin{proposition}\label{prop:BC}
Let $F\in \mathbb{D}^{1,2}(K)$ be such that $\E\big[ \| DF \|^4_{\fH\otimes K}
\big] <\infty$
and $\E[F] = 0$.
  Let $Z$ be a centered Gaussian
random variable on $K$ with covariance operator $S_Z$.
Then,

\noi
\begin{align}\label{MSBC}
d_2(F, Z) \leq \frac{1}{2} \big\|  \jb{DF, -DL^{-1}F}_\fH - S_Z
\big\|_{L^2(\O; \textup{HS})},
\end{align}

\noi
where  $\textup{HS}$ denotes the space of Hilbert-Schmidt operators on $K$;
see \eqref{HS}.


\end{proposition}

Note that in Theorem 3.2 of \cite{BC20},  Bourguin and Campese
established their bound \eqref{MSBC} under the additional assumption
that $Z$ is non-degenerate {\rm(}see Definition \ref{def1}{\rm)}.
In the following, we present an alternative proof using
the smart-path method, and we do not need to assume the
non-degeneracy.  Let us first explain the terms in \eqref{MSBC}.
\begin{itemize}

\item[(i)]  The inner product $\jb{DF, -DL^{-1}F}_\fH$ lives in $K\otimes K$,
and we can  view it as an operator mapping from $K$ to $K$.
More precisely, for $\phi_1, \phi_2\in K$,

\noi
\begin{align}
\label{conv1}
 ( \phi_1\otimes \phi_2) (k):= \phi_1\langle \phi_2, k\rangle_K
 \end{align}
 for any $k\in K$. It follows that

 \noi
 \begin{align}\label{conv1b}
 \| \phi_1 \otimes \phi_2\|_{\rm HS} = \| \phi_1\|_K \| \phi_2\|_K,
 \end{align}
 as one can easily verify.
  It is easy to see from the convention \eqref{conv1} and Parseval's identity
 with the decomposition \eqref{decomp1}
  that

\noi
\begin{align}
\begin{aligned}
  \big\| \jb{DF, -DL^{-1}F}_\fH \big\|_{\rm HS}^2
&=   \big\|  \sum_{i,j\geq 1} \jb{DF_i, -DL^{-1}F_j}_\fH \, k_i \otimes k_j
\big\|_{\rm HS}^2
\quad \text{with $F_j = \jb{F, k_j}_K$}\\
&=  \sum_{i, j\geq 1}  \big|  \langle  DF_i,   - DL^{-1}F_j  \rangle_\fH
\big|^2
\leq  \sum_{i, j\geq 1}   \|  DF_i \|^2_{\fH}   \|  DL^{-1}F_j  \|^2_\fH  \\
&= \| DF \|_{\fH\otimes K}^2 \| DL^{-1} F \|_{\fH\otimes K}^2 .
\end{aligned}
\label{BC1}
\end{align}

\noi
It follows from \eqref{BC1}, Cauchy-Schwarz, and \eqref{rep4b}
that

\noi
\begin{align*}
\E\Big[  \big\| \jb{DF, -DL^{-1}F}_\fH \big\|_{\rm HS}^2 \Big]
\leq  \E\big[ \| DF \|_{\fH\otimes K}^4 \big] < \infty.
  \end{align*}
This implies particularly that $ \jb{DF, -DL^{-1}F}_\fH $ is indeed an
Hilbert-Schmidt
operator on $K$ almost surely,
and the bound in \eqref{MSBC} is finite for $F\in \mathbb{D}^{1,2}(K)$
with  $\E\big[ \| DF \|^4_{\fH\otimes K} \big] <\infty$.

\item[(ii)]  It is not difficult to show that
\begin{align}  \label{exp_G}
\E\big[  \jb{DF, -DL^{-1}F}_\fH  \big] = S_F,
\end{align}
where $S_F$ denotes the covariance operator of $F$.
Indeed, denoting the LHS of \eqref{exp_G} by $S'_F$,
we have, with the convention \eqref{conv1}, that

\noi
\begin{align}
\begin{aligned}
\jb{ S'_F  k_i, k_j }_K
&= \E\big[  \jb{DF_j, -DL^{-1}F_i}_\fH  \big] \\
&=  \E\big[  F_j ( -\dl DL^{-1}F_i ) \big] = \E\big[  F_j F_i  \big],
\end{aligned}
 \label{exp_G2}
\end{align}
which is equal to $\jb{S_F k_i, k_j}_K$. Therefore, the equality \eqref{exp_G}
is verified.

\item[(iii)]  In general, $S_F$ may be different from $S_Z$.
Then, we can further bound the RHS of \eqref{MSBC} by
using the triangle inequality:

\noi
\begin{align}\label{MSBC2}
d_2(F, Z)
\leq \frac{1}{2} \big\|  \jb{DF, -DL^{-1}F}_\fH - S_F \big\|_{L^2(\O;
\textup{HS})}
+ \frac{1}{2} \| S_F - S_Z \|_{\rm HS}.
\end{align}

\end{itemize}

\begin{remark}\label{rem:MSdag}\rm
Under the assumptions of Proposition \ref{prop:BC},
\eqref{MSBC2} and \eqref{d2dag_cov} imply
\begin{align}\label{MSdag}
 d_2^\dagger(F,Z)
 &\leq \frac12
 \big\|\jb{DF,-DL^{-1}F}_\fH-S_F
       \big\|_{L^2(\Omega;\textup{HS})}
 \notag\\
 &\qquad +\frac12\|S_F-S_Z\|_{\rm HS}
       +\big|{\rm tr}(S_F)-{\rm tr}(S_Z)\big|.
\end{align}
Consequently, each second-order bound below also holds
with $d_2^\dagger$ on the left-hand side after adding
$|{\rm tr}(S_F)-{\rm tr}(S_Z)|$ to its right-hand side.
When $S_F=S_Z$, no additional term is needed.
This follows directly from the definition of $d_2^\dagger$;
the Malliavin-Stein estimate in Proposition \ref{prop:BC}
follows from the bound of Bourguin and Campese \cite{BC20}.
\end{remark}

\begin{proof}[Proof of Proposition \ref{prop:BC}]
Recall the definition of $d_2$ from \eqref{d2}
and fix an orthonormal basis $\{k_i: i\geq 1\}$ of $K$.
We first write using continuity that
\begin{align*}
d_2(F, Z)
&= \sup_{h} \sup_{N\geq 1}  \E[ h\circ P_N(F)] -  \E  [ h\circ P_N(Z) ] ,
\end{align*}
where the first supremum runs over all functions
$h$ with
$
\max_{j=1,2} \sup_{x\in K}  \| \nb^j h(x) \|_{K^{\otimes j}} \leq 1$,
and $P_N$ denotes the projection operator onto
 the subspace of $K$ that is generated by $\{k_i: i\leq N \}$.

 With $F_j  = \langle F, k_j\rangle_K$ and $Z_j =  \langle Z, k_j\rangle_K$,
 we can write
\begin{align*}
   \E[ h\circ P_N(F)] -  \E  [ h\circ P_N(Z) ]
   = \E\bigg[ h\Big(  \,  \sum_{j=1}^N F_j k_j \, \Big)\bigg] - \E\bigg[
h\Big(  \,  \sum_{j=1}^N Z_j k_j \, \Big)\bigg],
\end{align*}
 which can be viewed as
\begin{align*}
 \E\big[ \wt{h}(F_1, ..., F_N) \big] -  \E\big[ \wt{h}(Z_1, ..., Z_N) \big]
\end{align*}
 with  $\wt{h}(u_1, ...,u_N) =h\big(   \sum_{j=1}^N u_j k_j  \big) $
 as a function from $\R^N$ to $\R$.

 Following the same argument as in \cite[Theorem 6.1.2]{NP12}
 (see also \cite[Proposition 5.13]{FHMNP23}),
 we can obtain with $\mathbf{F}_N = (F_1, ... , F_N)$ and $\mathbf{Z}_N =
(Z_1, ... , Z_N)$
 that

 \noi
 \begin{align*}
&  -  \E\big[ \wt{h}(\mathbf{F}_N) \big] + \E\big[ \wt{h}(\mathbf{Z}_N) \big]
\\
  &\qquad = \frac{1}{2} \int_0^1 dt \sum_{i,j=1}^N
  \E\bigg[  \frac{\partial \wt{h}}{\partial x_i  \partial x_j} (
\sqrt{1-t}\mathbf{F}_N + \sqrt{t} \mathbf{Z}_N)
  \big( C_{i,j} - \langle DF_j, -DL^{-1}F_i \rangle_\fH \big) \bigg],
  \end{align*}

\noi
where $C_{i, j}: = \jb{S_Z k_j, k_i}_K$. It is not difficult to see that
the above expression coincides with
\begin{align*}
   \frac{1}{2} \int_0^1 dt \,
  \E\Big[ \big\langle  \nabla^2(h\circ P_N)
  ( \sqrt{1-t}F + \sqrt{t}Z),
  S_Z -   \langle DF, -DL^{-1}F \rangle_\fH \big\rangle_{\rm HS} \Big],
\end{align*}
which is bounded by
\begin{align*}
\frac{1}{2} \E\Big[ \big\|  S_Z -   \langle DF, -DL^{-1}F \rangle_\fH
\big\|_{\rm HS} \Big].
\end{align*}
 Cauchy-Schwarz gives \eqref{MSBC}, and the triangle inequality
 gives \eqref{MSBC2}.
 \qedhere

\end{proof}

For centered Gaussian random variables $G_1,G_2$ on $K$,
the same independent Gaussian interpolation gives
\begin{align}\label{Gaussian_comp}
 d_2(G_1,G_2)\leq\frac12\|S_{G_1}-S_{G_2}\|_{\rm HS};
\end{align}
see also \cite[Proposition 5.13]{FHMNP23}.
This bound depends only on the two Gaussian laws.

\bigskip

For the second-order bounds, assume that
$F\in\mathbb D^{2,4}(K)$ is centered and that $Z$
is a centered Gaussian random variable with $S_Z=S_F=:S$.
Write $F_i=\jb{F,k_i}_K$ and
$S_{ij}=\jb{Sk_j,k_i}_K$.
Then
\begin{align}\label{MSBC3}
 \big\|\jb{DF,-DL^{-1}F}_\fH-S\big\|_{L^2(\O;{\rm HS})}^2
 &=\E\sum_{i,j\geq1}Y_{ij}^2,
\end{align}
where
\begin{align}\label{MSBC3b}
 Y_{ij}:=\jb{DF_i,-DL^{-1}F_j}_\fH-S_{ij}.
\end{align}
By \eqref{exp_G}, $Y_{ij}$ is centered.
The assumption $F\in\mathbb D^{2,4}(K)$ allows us to apply
both the Malliavin product rule and the Gaussian
Poincar\'e inequality, which give
\begin{align}\label{MSBC3c}
 \E[Y_{ij}^2]
 &\leq2\E\big[\|\jb{D^2F_i,-DL^{-1}F_j}_\fH\|_\fH^2\big]
 +2\E\big[\|\jb{DF_i,-D^2L^{-1}F_j}_\fH   \|_\fH^2\big].
\end{align}
For the first term, \eqref{rep4} and Jensen's inequality yield
\begin{align}\label{MSBC3d}
 \E\big[\|\jb{D^2F_i,-DL^{-1}F_j}_\fH\|_\fH^2\big]
 &\leq\int_0^\infty e^{-t}
 \E\big[\|\jb{D^2F_i,P_tDF_j}_\fH\|_\fH^2\big]\,dt.
\end{align}

To estimate the sum over $i,j$, identify $D^2F$ with the
Hilbert-Schmidt operator $A:\fH\otimes K\to\fH$ given by
\begin{align*}
 A\bigg(\sum_i h_i\otimes k_i\bigg)
 &=\sum_i A_i h_i,
 \qquad 
  A_i h:=\jb{D^2F_i,h}_\fH.
\end{align*}
Each $A_i$ is self-adjoint, since $D^2F_i$ is symmetric.
Consequently,
\begin{align}\label{MSBC3e}
 \sum_{i,j\geq1}
 \|\jb{D^2F_i,P_tDF_j}_\fH\|_\fH^2
 &=\sum_{j\geq1}\|A^*(P_tDF_j)\|_{\fH\otimes K}^2
 \notag\\
 &\leq\|A\|_{\rm op}^2
          \|P_tDF\|_{\fH\otimes K}^2.
\end{align}
Here $\|A\|_{\rm op}
=\|D^2F\|_{\fH\otimes K\to\fH}$.

For the second term in \eqref{MSBC3c}, use the
second-order counterpart of \eqref{rep4},
\begin{align*}
 -D^2L^{-1}F
 &=\int_0^\infty e^{-2t}P_tD^2F\,dt.
\end{align*}
Jensen's inequality, with the probability measure
$2e^{-2t}\,dt$, and the operator estimate above imply
\begin{align}\label{MSBC3f}
 \sum_{i,j\geq1}
 \|\jb{DF_i,-D^2L^{-1}F_j}_\fH\|_\fH^2
 &\leq\frac12\int_0^\infty e^{-2t}
 \|P_tD^2F\|_{\fH\otimes K\to\fH}^2
 \|DF\|_{\fH\otimes K}^2\,dt.
\end{align}
Put
\begin{align*}
 a&:=\big(\E[\|D^2F\|_{\fH\otimes K\to\fH}^4]\big)^{1/4}
 \quad
 {\rm and}
 \quad
 b:=\big(\E[\|DF\|_{\fH\otimes K}^4]\big)^{1/4}.
\end{align*}
Mehler's formula and Jensen's inequality give
\begin{align*}
 \E[\|P_tD^2F\|_{\fH\otimes K\to\fH}^4]
 &\leq a^4.
\end{align*}
Thus \eqref{OU1}, Cauchy-Schwarz, and
\eqref{MSBC3d}--\eqref{MSBC3f} give, respectively,
\begin{align*}
 \sum_{i,j\geq1}
 \E\big[\|\jb{D^2F_i,-DL^{-1}F_j}_\fH\|_\fH^2\big]
 &\leq a^2b^2,\\
 \sum_{i,j\geq1}
 \E\big[\|\jb{DF_i,-D^2L^{-1}F_j}_\fH\|_\fH^2\big]
 &\leq\frac14a^2b^2.
\end{align*}
Combining these estimates with \eqref{MSBC} and
\eqref{MSBC3}--\eqref{MSBC3c}, we obtain
$d_2(F,Z)\leq\sqrt{\frac58}\,ab\leq ab$.
This proves the following result.

\begin{theorem}\label{thm1}
Let $F\in\mathbb D^{2,4}(K)$ be centered, and let $Z$
be a centered Gaussian random variable on $K$
with the same covariance operator as $F$.
Then
\begin{align*}
 d_2(F,Z)
 &\leq\big(\E[\|D^2F\|_{\fH\otimes K\to\fH}^4]\big)^{1/4}
       \big(\E[\|DF\|_{\fH\otimes K}^4]\big)^{1/4},
\end{align*}
where the first norm is the operator norm from
$\fH\otimes K$ to $\fH$.
\end{theorem}

We next relate this operator norm to a contraction norm.
The following lemma extends the scalar inequality in
\cite[Lemma 5.3.9]{NP12}; see also \cite{NPR09}.

\begin{lemma}[Random contraction inequality]\label{RCI}
For $F\in\mathbb D^{2,2}(K)$, almost surely,
\begin{align}\label{RCI1b}
 \|D^2F\|_{\fH\otimes K\to\fH}^4
 &\leq\|D^2F\otimes_1D^2F
          \|_{\fH^{\otimes2}\otimes K^{\otimes2}}^2.
\end{align}
The contraction $\otimes_1$ is taken over $\fH$.
\end{lemma}

\begin{proof}
Fix an realization  for which $D^2F\in\fH^{\odot2}\otimes K$.
Use the operators $A$ and $A_i$ introduced above.
Let $(h_a)_{a\geq1}$ and $(e_i)_{i\geq1}$
be orthonormal bases of $\fH$ and $K$, respectively.
Write
\begin{align*}
 D^2F
 &=\sum_{i,a,b\geq1}
   c_{i,a,b}\,h_a\otimes h_b\otimes e_i,
\qquad
 c_{i,a,b}
 =\langle A_i h_b,h_a\rangle_\fH.
\end{align*}
Since $D^2F$ is symmetric in its $\fH$ coordinates,
$c_{i,a,b}=c_{i,b,a}$ and $A_i^*=A_i$.
The $(i,j)$ block of $A^*A$ is therefore $A_iA_j$.
Using the orthonormal basis
$(h_b\otimes e_j)_{b,j\geq1}$ of $\fH\otimes K$,
we obtain
\begin{align*}
 \|A^*A\|_{\rm HS}^2
 &=\sum_{j,b\geq1}
   \|A^*A(h_b\otimes e_j)\|_{\fH\otimes K}^2\\
 &=\sum_{i,j,b\geq1}\|A_iA_jh_b\|_\fH^2
 =\sum_{i,j\geq1}\|A_iA_j\|_{\rm HS}^2.
\end{align*}
By definition of contractions \eqref{contra1}, 
we have
\begin{align*}
 D^2F\otimes_1D^2F
 =\sum_{i,j,a,b\geq1}
 \left(\sum_{c\geq1}c_{i,a,c}c_{j,b,c}\right)
 h_a\otimes h_b\otimes e_i\otimes e_j.
\end{align*}
With  $ c_{j,b,c}
 =\langle A_jh_c,h_b\rangle_\fH
 =\langle A_jh_b,h_c\rangle_\fH$,
we get  the coefficient of
$h_a\otimes h_b\otimes e_i\otimes e_j$
is
\begin{align*}
 \sum_{c\geq1}c_{i,a,c}c_{j,b,c}
 &=\sum_{c\geq1}
   \langle A_i h_c,h_a\rangle_\fH
   \langle A_j h_b,h_c\rangle_\fH\\
 &=\left\langle
   A_i\left(\sum_{c\geq1}
   \langle A_jh_b,h_c\rangle_\fH h_c\right),
   h_a\right\rangle_\fH\\
 &=\langle A_iA_jh_b,h_a\rangle_\fH,
\end{align*}
where the last equality follows from
the orthonormal expansion of $A_jh_b$.
Then, we deduce from 
Parseval's identity that
\begin{align*}
 \|D^2F\otimes_1D^2F
   \|_{\fH^{\otimes2}\otimes K^{\otimes2}}^2
 &=\sum_{i,j,a,b\geq1}
   |\langle A_iA_jh_b,h_a\rangle_\fH|^2\\
 &=\sum_{i,j\geq1}\|A_iA_j\|_{\rm HS}^2
 =\|A^*A\|_{\rm HS}^2.
\end{align*}

\noi
Therefore,
\begin{align*}
 \|D^2F\|_{\fH\otimes K\to\fH}^4
 &=\|A^*A\|_{\rm op}^2
 \leq\|A^*A\|_{\rm HS}^2,
\end{align*}
which proves \eqref{RCI1b}.
\end{proof}

Theorem \ref{thm1} and Lemma \ref{RCI} now yield
our contraction form of the second-order bound.

\begin{theorem}\label{thm2}
Let $F\in\mathbb D^{2,4}(K)$ be centered, and let $Z$
be a centered Gaussian random variable on $K$
with the same covariance operator as $F$.
Then
\begin{align}\label{2ndGPI}
 d_2(F,Z)
 &\leq\big(\E[\|D^2F\otimes_1D^2F
         \|_{\fH^{\otimes2}\otimes K^{\otimes2}}^2]\big)^{1/4}
       \big(\E[\|DF\|_{\fH\otimes K}^4]\big)^{1/4}.
\end{align}
\end{theorem}

\section{Main results and applications}\label{SEC3}

In the spirit of the work \cite{Vid20} by the first author,
we first present in Section \ref{SEC31} an improved version of the
second-order
Gaussian Poincar\'e inequality \eqref{2ndGPI}
in the particular setting that
the Hilbert spaces $\fH = L^2( A, \mathscr{B}(A),  \mu )$ and $K=L^2( E,
\mathscr{B}(E),  \nu )$,
where $A$ and $E$ are Polish spaces equipped with their Borel
$\s$-algebras and positive, $\s$-finite measures
$\mu$ and $\nu$, respectively.
The arguments below also allow these measures to have atoms.

We would like to point it out that
 there could be
other representation of $\fH$ that can lead to a different form
of the improved second-order Gaussian Poincar\'e inequality.
Motivated by applications in stochastic partial differential equations
(e.g.,  \cite[Proposition 1.8]{BNQSZ}),
we present in Section \ref{SEC32} another variant of
the functional second-order Gaussian Poincar\'e inequality.

\subsection{When \texorpdfstring{$\fH$}{H} is an \texorpdfstring{$L^2$}{L2}
space} \label{SEC31}

Let us first fix a few more notations.
When $\fH = L^2(A, \mathscr{B}(A),  \mu)$, the Malliavin derivative
$DF$ of a real-valued Malliavin differentiable random variable
is a random element in $ L^2(A, \mathscr{B}(A),  \mu)$:

\noi
\begin{align}\label{conv2}
x\in A \mapsto D_xF \in \R.
\end{align}

\noi
Similarly, we write

\noi
\begin{align}\label{conv3}
(x, y)\in A^2 \mapsto D^2_{x, y} G
\end{align}

\noi
to denote the second-order Malliavin derivative of $G$,
whenever it is   defined.
Also, we write

\noi
\begin{align}\label{conv4}
(D^2F \otimes_1 D^2G)(x, y)
:= \int_A  (D^2_{x, z} F) (D^2_{y, z} G)   \mu(dz),
\end{align}
which is consistent with  \eqref{contra1}.

Throughout  this Section \ref{SEC31},  $\fH=L^2(A,\mathscr{B}(A),\mu)$ and
$K=L^2(E, \mathscr{B}(E),  \nu )$.

\begin{theorem}\label{thm_imp1}
Let $F\in\mathbb D^{2,4}(K)$ be centered, with
covariance operator $S_F$.
Let $Z$ be a centered Gaussian random variable on $K$
with covariance operator $S_Z$.
Then,

 \noi
\begin{align}\label{mb2}
\begin{aligned}
d_2(F, Z)&\le  \frac{\sqrt{3}}{2}
\bigg( \int_{E^2}  \nu(dr_1) \nu(dr_2)
   \int_{A^2}  \big\| D D_xF(r_1) \otimes_1 D D_zF(r_1) \big\|_{L^2(\O)}    \\
 &\qquad \times   \big\|   D_x F(r_2)      D_z F(r_2)   \big\|_{L^2(\O)}
\,\mu(dx)  \mu(dz)\bigg)^{\frac{1}{2}}
  + \frac{1}{2} \| S_F - S_Z \|_{\rm HS}
\end{aligned}
\end{align}
and moreover, we also have

 \noi
\begin{align}\label{mb1}
\begin{aligned}
d_2(F, Z)&\le  \frac{\sqrt{3}}{2}
\bigg( \int_{E^2}  \int_{A^3}
 \| D^2_{x, y} F(r_1) \|_{L^4(\Omega)}  \|  D^2_{z, y} F(r_1)
\|_{L^4(\Omega)}   \| D_x F(r_2)  \|_{L^4(\Omega)}
  \\
    &\qquad
    \times   \|  D_z F(r_2)  \|_{L^4(\Omega)}
\mu(dx)\mu(dy)\mu(dz) \nu(dr_1) \nu(dr_2)  \bigg)^{\frac{1}{2}}
 + \frac{1}{2} \| S_F - S_Z \|_{\rm HS}.
\end{aligned}
\end{align}

\end{theorem}

The bound \eqref{mb1} is cruder than \eqref{mb2}, and it can also be
derived as a particular case of the bound in Theorem \ref{thm_imp2}.

\begin{proof}[Proof of Theorem \ref{thm_imp1}]
Let $Z_F$ be centered Gaussian with covariance $S_F$.
By the triangle inequality and Gaussian comparison
\eqref{Gaussian_comp}, it suffices to prove the bounds when
$S_F=S_Z=S$.
The assumption $F\in\mathbb D^{2,4}(K)$ ensures that
$\langle DF,-DL^{-1}F\rangle_\fH$
belongs to $\mathbb D^{1,2}(K^{\otimes2})$,
so the product rule and Poincar\'e inequality below apply.
From \eqref{MSBC}, \eqref{conv1}, and  \eqref{conv1b}, we know that
\begin{align}\label{bound1}
d_2(F, Z) &\leq \frac{1}{2} \big\|  \jb{DF, -DL^{-1}F}_\fH - S
\big\|_{L^2(\Omega; K^{\otimes2})}\\
&= \frac{1}{2} \sqrt{\E \[\big\|  \jb{DF, -DL^{-1}F}_\fH - S
\big\|_{K^{\otimes2}}^2\] }\notag\,.
\end{align}

\noi
Since $\E \[ \jb{DF, -DL^{-1}F}_\fH\]=S$,
we can deduce from   the Gaussian Poincar\'e inequality \eqref{GPI}
(applied to $K^{\otimes 2}$-valued random variables)
that

\noi
\begin{align}\label{bound2a}
\begin{aligned}
&\E\[ \big\|  \jb{DF, -DL^{-1}F}_\fH - S \big\|_{K^{\otimes2}}^2\]
\leq \E\[\big\|  D\jb{DF, -DL^{-1}F}_\fH \big\|_{\fH \otimes K^{\otimes2}}^2\]
\\
&=\E\[\big\|  \jb{D^2F, -DL^{-1}F}_\fH
+\jb{DF, -D^2L^{-1}F}_\fH \big\|_{\fH \otimes K^{\otimes2}}^2\]\\
&\le 2\,  \E\[\big\|  \jb{D^2F, -DL^{-1}F}_\fH
\big\|_{\fH \otimes K^{\otimes2}}^2\]
+2\, \E\[\big\|\jb{DF, -D^2L^{-1}F}_\fH \big\|_{\fH \otimes K^{\otimes2}}^2\].
\end{aligned}
\end{align}

\noi
Recalling \eqref{rep4},  we can first write

\noi
\begin{align}\label{bound2b}
\jb{D^2F, -DL^{-1}F}_\fH
&=\jb{D^2F, \int_0^\infty e^{-t} P_t DF dt}_\fH,
\end{align}

\noi
which is a random vector in $\fH\otimes K^{\otimes 2}$.
For $r_1, r_2\in E$, we can write by using Mehler's formula that

\noi
\begin{align*}
\jb{D^2F(r_1), -DL^{-1}F(r_2) }_\fH
&=  \int_0^\infty e^{-t} \int_A  D D_xF(r_1)
 P_t (D_x F(r_2))\mu(dx)dt  \\
&=   \int_0^\infty e^{-t} \int_A  D D_xF(r_1) \E\big[   [ D_x F(r_2)]^t | W
\big] \mu(dx)  dt \\
&=  \int_0^\infty e^{-t} \E\bigg(  \int_A D D_xF(r_1)    [ D_x F(r_2)]^t
\mu(dx)  \big| W \bigg)
dt ,
\end{align*}

\noi
where we used the notation
$X^t = \Psi(W^t)$
for  $X = \Psi(W)$
 a real-valued random variable that is $\s\{W\}$-measurable.
Then, using Jensen's inequality and Minkowski's inequality,
we get

\noi
\begin{align*}
&\big\|  \jb{D^2F, -DL^{-1}F}_\fH  \big\|_{\fH \otimes K^{\otimes2}}^2   \\
& \leq \int_0^\infty  dt \,   e^{-t}  \int_{E^2}  \nu(dr_1) \nu(dr_2)
\E\bigg(  \Big\|  \int_A D D_xF(r_1)    [ D_x F(r_2)]^t  \mu(dx) \Big\|_\fH^2
\big| W \bigg)
\\
&=
 \int_0^\infty  dt \,   e^{-t}  \int_{E^2}  \nu(dr_1) \nu(dr_2) \\
 &\qquad \times \E\bigg(  \int_{A^3} D_y D_xF(r_1)D_y D_zF(r_1)
   [ D_x F(r_2)]^t   [ D_z F(r_2)]^t  \mu(dx)  \mu(dz) \mu(dy)   \big| W
\bigg) \\
  &=  \int_0^\infty  dt \,   e^{-t}  \int_{E^2}  \nu(dr_1) \nu(dr_2)
   \int_{A^3} D_y D_xF(r_1)D_y D_zF(r_1)    \\
 &\qquad \times
  \E\big(   [ D_x F(r_2)]^t   [ D_z F(r_2)]^t   |W \big) \,\mu(dx)  \mu(dz)
\mu(dy) \\
  &=  \int_0^\infty  dt \,   e^{-t}  \int_{E^2}  \nu(dr_1) \nu(dr_2)
   \int_{A^2}  \big[  D D_xF(r_1) \otimes_1 D D_zF(r_1) \big]    \\
 &\qquad \times  P_t \big(   D_x F(r_2)      D_z F(r_2)   \big) \,\mu(dx)
\mu(dz).
\end{align*}

 \noi
Therefore,
it follows from the Cauchy-Schwarz inequality and \eqref{OU1}  that

\noi
\begin{align*}
&  \E\[\big\|  \jb{D^2F, -DL^{-1}F}_\fH
\big\|_{\fH \otimes K^{\otimes2}}^2\]  \\
&\qquad  \leq
  \int_{E^2}  \nu(dr_1) \nu(dr_2)
   \int_{A^2}  \big\| D D_xF(r_1) \otimes_1 D D_zF(r_1) \big\|_{L^2(\O)}    \\
 &\qquad\qquad \times   \big\|   D_x F(r_2)      D_z F(r_2)   \big\|_{L^2(\O)}
\,\mu(dx)  \mu(dz).
 \end{align*}
In the same way, we can deal with
$ \E\big[ \|\jb{DF, -D^2L^{-1}F}_\fH \big\|_{\fH \otimes K^{\otimes2}}^2\big]$
by using $-D^2L^{-1}F = \int_0^\infty e^{-2t} P_t D^2 F dt$ instead of
\eqref{rep4}:
\noi
\begin{align*}
&  \E\[\big\|  \jb{DF, -D^2L^{-1}F}_\fH
\big\|_{\fH \otimes K^{\otimes2}}^2\] \\
&\quad
 \leq  \frac{1}{4} \int_{E^2}  \nu(dr_1) \nu(dr_2)
   \int_{A^2}  \big\| D D_xF(r_1) \otimes_1 D D_zF(r_1) \big\|_{L^2(\O)}    \\
 &\qquad \times   \big\|   D_x F(r_2)      D_z F(r_2)   \big\|_{L^2(\O)}
\,\mu(dx)  \mu(dz).
 \end{align*}
Here the factor $1/4$ comes from
$(\int_0^\infty e^{-2t}dt)^2$.
If $I$ denotes the integral in \eqref{mb2},
\eqref{bound2a} is therefore bounded by
$2I+I/2\leq3I$.
This proves \eqref{mb2}; Cauchy--Schwarz and Minkowski
give \eqref{mb1}.
\qedhere

 \end{proof}

 \noi
 $\bul$ {\bf Application $\I$: Functional Breuer-Major type CLTs.}
Let $X=\{X(h): h \in \fH\}$ be an isonormal Gaussian process over the real
separable Hilbert space
$\fH=L^2\(\R,\mathscr{B}(\R),dx\)$.
Let $\{h_t: t\in\R\}\subset \fH$ be
such that
\begin{align*}
\langle h_t,h_s\rangle_\fH  = \int_\R h_t(x) h_s(x)  dx
=:\varrho(t-s)
\end{align*}
 with $\varrho(0)=1$
 and

\noi
\begin{align}\label{BMc1}
\int_\R|\varrho(t)|  dt< \infty.
\end{align}

\noi
Define $Y_t=X(h_t)$ for $t\in\R$.
Then $Y=(Y_t)_{t\in\R}$ is a centered stationary
Gaussian process.
Throughout, we assume that
$(\omega,t)\mapsto Y_t(\omega)$ is jointly measurable.
This assumption implies that $Y$ is mean-square
continuous, or equivalently, that its covariance
function $\rho$ is continuous at zero;
see \cite[Section 5.2]{MZ26}.

Let $f:\mathbb{R}\rightarrow\mathbb{R}$ be   of class $C^2$
such that
$f(N)$ and $ | f''(N)|^4$ have finite expectations
  with $N\sim\mathcal{N}(0,1)$.\footnote{This
 implies that both  $f(N)$ and $f'(N)$ have finite fourth moments by a general
 Gaussian  Poincar\'e inequality; see, e.g., \cite[Proposition 3.1]{NPR09}.
 }
Next,  we define

\noi
\begin{align}\label{def_FR}
F_T(r):=\frac{1}{\sqrt{T}} \int_{-rT}^{rT} \big(  f (Y_t )-\E [f (Y_t ) ]
\big)\,dt,   \,\,  r\in[0,1],
\end{align}
and we view $F_T$ as a random element of $K = L^2([0,1])$.
Let us first explain the normalization $\frac{1}{\sqrt{T}}$ in \eqref{def_FR},
which is quite standard in the context of Breuer-Major theorem; see, e.g.,
\cite[Chapter 7]{NP12}.
Since $f$ is square-integrable with respect to the standard Gaussian measure
on $\R$,
we have the following Hermite expansion
\begin{align*}
f (Y_t )-\E [f (Y_t ) ] = \sum_{q=1}^\infty c_q H_q(Y_t),
\end{align*}

\noi
where $H_q(x) = (-1)^q e^{\frac{x^2}{2}} \tfrac{d^q}{dx^q}(
e^{-\frac{x^2}{2}}) $ denotes
the $q$-th Hermite polynomial and $c_q = \tfrac{1}{q!} \E[ f(N) H_q(N) ]$ for
$q\in\N_{\geq 0}$.\footnote{The first few
Hermite polynomials are given by $H_0(x) =1$, $H_1(x) =x,$ $H_2(x) = x^2-1$,
and $H_{p+1}(x) =  xH_p(x) - p H_{p-1}(x)$ for $p\in\N_{\geq 2}$.
Note that $\{ \tfrac{1}{\sqrt{q!}} H_q: q\in \N_{\geq 0} \}$ is
an orthonormal basis for $L^2(\R,  \tfrac{1}{\sqrt{2\pi} } e^{-\frac{x^2}{2}}
dx )$.} Then,
using orthogonality relation of Hermite polynomials (see, e.g.,
\cite[Proposition 1.4.2]{NP12}), we write

\noi
\begin{align}
C_T(r_1, r_2):= \E\[F_T(r_1)F_T(r_2)\]
&=\frac{1}{T}  \int_{-r_1T}^{r_1T}  \int_{-r_2T}^{r_2T}
\Cov\big[  f (Y_t )   ,     f (Y_s )   \big] dt ds \notag \\
&= \frac{1}{T}  \int_{-r_1T}^{r_1T}  \int_{-r_2T}^{r_2T}
  \sum_{q=1}^\infty c_q^2  q! \rho^q(t-s)   dt ds   \notag  \\
  &\xrightarrow{T\to+\infty} 2  \min\{ r_1, r_2\}   \sum_{q=1}^\infty c_q^2
q! \int_\R \rho^q(t)   dt \notag \\
  & = : C_\infty(r_1, r_2), \label{limc1}
\end{align}
which follows from the dominated convergence theorem with \eqref{BMc1} and the
fact that
$|\rho(t)| \leq 1$. Alternatively, we can rewrite the above limit
\eqref{limc1} as

\noi
\begin{align}\label{limc2}
C_\infty(r_1, r_2) = 2  \min\{ r_1, r_2\}   \int_\R \Cov\big[  f (Y_t )   ,
f (Y_0 )   \big] dt.
\end{align}
Taking $r_1 = r_2$ in \eqref{limc1} or  \eqref{limc2}, we can see that
the above integral  $ \int_\R \Cov\big[  f (Y_t )   ,     f (Y_0 )   \big] dt$
is always nonnegative and finite, which is  also stated in \eqref{sigma_def}
below.
It is also easy to see that
\begin{align}
M_1 =  \int_{\R}  \big| \Cov\big[  f (Y_t )   ,     f (Y_0 )   \big] \big|  dt
 \leq \int_{\R}
  \sum_{q=1}^\infty c_q^2  q!  | \rho(t) | dt
  \leq \| \rho\|_{L^1(\R)} \E[ f(Y_0)^2]. \label{COV_int}
\end{align}

\begin{theorem} \label{thm_app1}
Let $F_T$ be defined as in \eqref{def_FR}.
Assume that
\begin{align}\label{condition}
\abs{h_t(x)}\leq g(t-x) ,
\end{align}

\noi
where
$g:\R\to\R_+$ is measurable and satisfies the following uniform bound:
 \begin{align}\label{uni_g}
  G_\star:=\sup_{u}\int_{\R}g(t+u)\,dt<\infty .
  \end{align}

\noi
Then, denoting by $\mathcal{B}_T$ a centered Gaussian
random variable on $K$ with the same covariance operator
as $F_T$, we have

\noi
\begin{align}\label{BM_r1}
d_2^\dagger\big(F_T,\mathcal B_T\big)
=d_2\big(F_T,\mathcal B_T\big)
\leq
\frac{C}{\sqrt{T}},
\end{align}

\noi
where the constant $C = \big(\sqrt{3}   \| f'(Y_0) \|_{L^4(\O)}   \| f''(Y_0)
\|_{L^4(\O)}  G_\star^3    \big) /2$
with $Y_0\sim\NN(0,1)$.
Let $B$ denote a standard Brownian motion on $[0,1]$.\footnote{We view $B$
as a random element
of $K = L^2([0,1])$.}
Then,

\noi
\begin{align}\label{BM_r1b}
d_2^\dagger\big(\mathcal B_T,\sigma B\big)
\xrightarrow{T\to+\infty}0,
\end{align}

\noi
where

\noi
\begin{align} \label{sigma_def}
\sigma : =  \bigg(2  \int_\R \Cov\big[  f (Y_t )   ,     f (Y_0 )   \big]
dt\bigg)^{\frac12} \in [0,\infty).
\end{align}

\noi
In particular, $d_2^\dagger(F_T,\sigma B)\to0$.
Consequently, $W_2(F_T,\sigma B)\to0$ and
$F_T$ converges weakly to $\sigma B$ in $L^2([0,1])$,
as $T\to+\infty$.

\end{theorem}

Note that $F_T$ has an infinite chaos expansion. In the work \cite{BC20},
the authors restricted their consideration to a single chaos (see Theorem 5.1
therein),
while our Theorem \ref{thm_app1}
illustrates the power of the second-order Gaussian Poincar\'e inequality
and its flexibility in functional approximation.

\begin{proof}[Proof of Theorem \ref{thm_app1}]
The assumptions on $f$ imply that
$f,f',f''\in L^4(\R,\gamma)$, where $\gamma$
is the standard Gaussian measure.
Since $\|h_t\|_\fH=1$, the chain rule and
Minkowski's inequality give $F_T\in\mathbb D^{2,4}(K)$.
Since $F_T$ and $\mathcal B_T$ are centered and have
the same covariance operator, Remark \ref{rem:d2dag}
gives $d_2^\dagger(F_T,\mathcal B_T)=d_2(F_T,\mathcal B_T)$.
In view of the bound \eqref{mb1}, we need to compute the
$L^4(\Omega)$-norm of $D_x F_T(r)$ and $D^2_{x, y} F_T(r)$.
 Since
\begin{align*}
D_x F_T(r)&=\frac{1}{\sqrt{T}}\int_{-rT}^{rT}\, f'\left(Y_t\right)
h_t(x)\,dt\,,\\
D_{x,y}^2 F_T(r)&=\frac{1}{\sqrt{T}}\int_{-rT}^{rT}\, f''\left(Y_t\right)
h_t(x) h_t(y)\,dt .
\end{align*}
Then, using \eqref{condition} and Minkowski's inequality,
we get

\noi
\begin{align}  \label{D_bdd1}
\begin{aligned}
\|D_x F_T(r) \|_{L^4(\O)}
&\leq \frac{   \| f'(Y_0) \|_{L^4(\O)}    }{\sqrt{T}}\int_{-rT}^{rT} g(t-x) dt
\\
\|D^2_{x, y} F_T(r) \|_{L^4(\O)}
&\leq \frac{\| f''(Y_0) \|_{L^4(\O)}    }{\sqrt{T}}\int_{-rT}^{rT} g(t-x)
g(t-y)  dt ,
\end{aligned}
\end{align}

\noi
where $Y_0\sim\NN(0,1)$.
Therefore, plugging the above estimates \eqref{D_bdd1} into \eqref{mb1}
yields

\noi
\begin{align*}
&d_2(F_T, \mathcal{B}_T)\\
\quad & \leq \frac{\sqrt{3}   \| f'(Y_0) \|_{L^4(\O)}   \| f''(Y_0)
\|_{L^4(\O)}      }{2 T}
\bigg( \int_{[0,1]^2} dr_1 dr_2 \int_{\R^3} dxdydz \int_{-r_1 T} ^{r_1T}  dt_1
g(t_1-x)g(t_1-y)  \\
&\qquad \quad \times
 \int_{-r_1 T} ^{r_1T} dt_2 g(t_2-z)g(t_2-y)
   \int_{-r_2 T} ^{r_2T} dt_3 g(t_3-x)
 \int_{-r_2T} ^{r_2T} dt_4 g(t_4-z)   \bigg)^{\frac12} \\
 &=:  \frac{\sqrt{3}   \| f'(Y_0) \|_{L^4(\O)}   \| f''(Y_0) \|_{L^4(\O)}
}{2 T}   \mathbf{J}^{\frac12},
\end{align*}

\noi
where $\mathbf J$ is defined in an obvious manner.
Next, recalling the uniform bound \eqref{uni_g}, we perform
the integrations in the exact order of
$dt_3$, $dx$, $dt_1$, $dy$, $dt_2$, $dz$,  and $dt_4$ to get

\noi
\begin{align*}
\mathbf{J} \leq  \int_{[0,1]^2} dr_1 dr_2  (2r_2 T) G_\star^6 = G_\star^6 T,
\end{align*}

\noi
from which, we can easily obtain the desired bound \eqref{BM_r1}.

Finally, let us show \eqref{BM_r1b}.
By \eqref{limc1} and \eqref{COV_int},
\begin{align*}
 |C_T(r_1,r_2)|
 &\leq 2\min\{r_1,r_2\}M_1,\\
 C_T(r_1,r_2)&\to\sigma^2\min\{r_1,r_2\}.
\end{align*}
Dominated convergence gives
$\|S_{F_T}-S_{\sigma B}\|_{\rm HS}\to0$.
Applying \eqref{Gaussian_comp} to the Gaussian pair
$(\mathcal B_T,\sigma B)$ therefore yields
$d_2(\mathcal B_T,\sigma B)\to0$.
We also have
\begin{align}\label{BM_trace}
 \E[\|F_T\|_K^2]
 =\E[\|\mathcal B_T\|_K^2]
 =\int_0^1 C_T(r,r)dr
 \longrightarrow\frac{\sigma^2}{2}
 =\E[\|\sigma B\|_K^2].
\end{align}
Indeed, $0\leq C_T(r,r)\leq2rM_1$
and $C_T(r,r)\to\sigma^2r$.
Combining \eqref{BM_trace} with \eqref{d2dag}
proves \eqref{BM_r1b}.
The remaining assertions follow from the triangle inequality
and Proposition \ref{prop_d2mom}.
\qedhere

\end{proof}

 \noi
 $\bul$ {\bf Application $\II$: Shallow neural network.}
In a recent work \cite{BFF24},
Bordino, Favaro, and Fortini applied the second-order Gaussian Poincar\'e
inequality from \cite{Vid20} to derive quantitative central limit theorems
for Gaussian neural networks. In this section, we apply our functional version
of
the second-order Gaussian Poincar\'e
inequality to establish a functional central limit theorem with rate of
convergence.

For an {\bf activation} function $\tau: \R\to\R$,
consider the following (simple)
  fully connected feed-forward Gaussian shallow  neural network of  width $n$
  with one-dimensional input:\footnote{This neural network is shallow since
its depth is one.
  For relevant quantitative central limit theorem for deep neural network, we
refer
  interested readers to the recent work \cite{FHMNP23}}

 \noi
\begin{align}\label{NN1}
f_n(x)
=  \frac{1}{n^{1/2}}\sum_{j=1}^{n}w_j \tau \big( x w^{(0)}_j  \big),
\end{align}
where
\begin{itemize}
\item  $f_n(x)$ is the {\bf output} given the   {\bf input} $x \in \R$;

\item $\Theta=\{w_{j}^{(0)}, w_{j}\}_{j}$ is a collection of \emph{random
weights}
that are  independent and identically distributed as  $\NN(0,\sigma^{2}_{w})$.

\end{itemize}

In \cite[Theorem 4]{BFF24}, Bordino, Favaro, and Fortini considered a
particular
case where  $x =\s^2_w=1$ and $\tau\in C^2(\R)$ such that
\begin{align}\label{env_cond}
\abs{\frac{d^\ell}{dx^\ell} \tau(x) } \leq a + b |x|^\gamma \qquad \ell=0,
1,2,
\end{align}
for some $a, b, \gamma > 0$. In this setting, they established a
quantitative central limit theorem for the neural network
when the width tends to infinity:

\noi
\begin{align}\label{eq:thmBFF20}
d_{\rm M}\big(f_n(1), \NN(0, \s^2) \big)\leq \frac{K_{\sigma^{2}}}{n^{1/2}}
\end{align}
for some explicit constant $K_{\sigma^{2}}$,
where $\s^2 = \E[ \tau^2(Z) ]$ with $Z\sim\NN(0,1)$.
Here, $d_{\rm M}$ stands for any of the distributional metrics
(i) $1$-Wasserstein distance (ii) Kolmogorov distance (iii) total-variation
distance;
see, e.g., \cite[Appendix C]{NP12} for more details on these distributional
metrics.

In the following, we take $\sigma_w^2=1$
and apply our functional second-order
Gaussian Poincar\'e inequality to the shallow neural network \eqref{NN1}
with the input $x$ varying in $\R$. Let us first fix the Hilbert space $K =
L^2(\R, d\nu)$
(with $\nu$ to be determined),
in which $f_n(\bul)$ lives. Suppose the condition \eqref{env_cond}
holds, in particular, $|\tau(x)| \leq a + b |x|^\gamma$.
Then, we need to impose (sufficient) conditions on $\nu$ such that  $f_n\in
L^2(\R, d\nu)$ almost surely.
For this purpose, let us do some second moment computations:
\begin{align*}
\E \int_\R  |  f_n(x)|^2 \nu(dx)
& =\frac{1}{n} \int_\R   \E \bigg[   \big| \sum_{j=1}^{n}w_j \tau \big( x
w^{(0)}_j  \big) \big|^2 \bigg] \nu(dx)
\quad\text{by Fubini} \\
&= \frac{1}{n} \int_\R   \sum_{j=1}^{n}    \E \big[ \tau^2  ( x w^{(0)}_j   )
\big]   \nu(dx) \quad\text{using independence}\\
&= \int_\R     \E \big[ \tau^2  ( x Z  ) \big]   \nu(dx) \quad\text{with
$Z\sim\NN(0,1)$} \\
&\leq  2\int_\R  \big( a^2 + b^2 |x|^{2\gamma}
 \E[ |Z|^{2\gamma}] \big) \nu(dx),
\end{align*}
where we used the condition $|\tau(x)| \leq a + b |x|^\gamma$
and the inequality $(x+y)^2 \leq 2x^2 + 2y^2$ in the last step.
It is then clear that if we assume
\begin{align}\label{assume_nu}
 \int_{\R} (1+ |x|^{2\gamma})\nu(dx) <\infty,
\end{align}
 we can get $f_n\in L^2(\Omega; K)$.  It is not difficult to get

 \noi
 \begin{align}\label{cov_fn}
\mathcal{C}(x,y):= \E[ f_n(x) f_n(y) ] =  \E[ \tau(xZ) \tau(yZ) ],
 \end{align}
 which does not depend on the width $n$, and $\mathcal{C}\in K^{\otimes 2}$
 under the assumption \eqref{assume_nu}.

\begin{theorem} \label{thm_NN}
Assume  \eqref{env_cond} holds for some positive constants $a, b,\gamma$.
Let $\nu$ be a finite   measure on $\R$ such that
\begin{align}\label{nu4ga}
 \int_{\R}  |x|^{2\gamma+4} \nu(dx) <\infty.
\end{align}
Let $\{B(x): x\in\R\}$ be a centered Gaussian field with covariance function
$\mathcal{C}$
as in \eqref{cov_fn}.
Recall the definition of $f_n$ from \eqref{NN1}. Viewing $f_n$ and $B$ as
random elements in $K= L^2(\R, d\nu)$,
we have

\noi
\begin{align}\label{CLT_NN}
d_2^\dagger(f_n,B)=d_2(f_n,B)\leq Cn^{-\frac12},
\end{align}

\noi
where $C$ is a finite constant that depends on $\nu, a, b,$ and $\gamma$.
In particular, $W_2(f_n,B)\to0$ and $f_n$
converges weakly to $B$ in $K$ as $n\to\infty$.

\end{theorem}

\begin{remark}\rm
(i)
When $\nu$ is a finite measure with the condition \eqref{nu4ga},
 the condition \eqref{assume_nu} holds.
 In particular, the Lebesgue measure restricted on any finite interval
satisfies  \eqref{nu4ga}.
 
 \smallskip

\noi
(ii) In the one-dimensional case, the authors of \cite{BFF24} obtained the
rate $n^{-1/2}$ in
total-variation distance and $1$-Wasserstein distance, while
the authors of \cite{FHMNP23} obtained the optimal rate $n^{-1}$
by combining a conditioning argument and Stein's method.
Whether an $n^{-1}$ rate holds for the present
infinite-dimensional $d_2$ approximation requires
additional analysis.

 \smallskip
 
\noi
(iii) In a recent work \cite{CMSV23}, the authors derived similar quantitative
functional CLT for shallow neural network with spherical inputs by using chaos
expansion of the activation
function. 
Note that restricting the domain of the activation function $\tau$ to a compact
manifold as the unit sphere
$\mathbb S^d$ is morally equivalent to taking 
$K=L^2(\mathbb S^d,dx)$ in our notation.

\end{remark}

\begin{proof}[Proof of Theorem \ref{thm_NN}]
By \eqref{cov_fn} and \eqref{assume_nu},
\begin{align*}
 \E[\|f_n\|_K^2]
 =\int_\R\mathcal C(x,x)\nu(dx)
 =\E[\|B\|_K^2]<\infty.
\end{align*}
Thus $d_2^\dagger(f_n,B)=d_2(f_n,B)$.
It suffices to prove the stated $d_2$ bound;
the convergence assertions then follow from
Proposition \ref{prop_d2mom}.
In this proof, we let $C$ denote a finite constant that depends on
$(\nu, a, b,  \gamma)$
and may vary from line to line. For instance, we rewrite the condition
\eqref{env_cond}
as
\begin{align}\label{env1}
 \sum_{\ell=0}^2 \abs{ \frac{d^\ell}{dx^\ell} \tau(x) } \leq C( 1+
|x|^\gamma),
\end{align}
from which we get for $G\sim\NN(0,1)$:
\begin{align}\label{env2}
 \|  \tau(rG) \|_{L^4(\Omega)}
   + \|  \tau'(rG) \|_{L^4(\Omega)}
   + \|  \tau''(rG) \|_{L^4(\Omega)}
\leq C( 1+ |r|^\gamma).
\end{align}

We now give the second-order Poincar\'e proof,
which illustrates Theorem \ref{thm_imp1}.
The random variable $f_n(r)$ is a function of finitely many
i.i.d. standard normals, realized as follows:
\begin{align*}
w_j = W( \ind_{[2j, 2j+1)  })
\quad
{\rm and}
\quad
w^{(0)}_j = W( \ind_{[2j+1, 2j+2)  })
\quad\text{for $j\geq 1$},
\end{align*}

where $W$ is an isonormal process over $\fH = L^2(\R_+, dx)$. In this setting
(with $S_{f_n} = S_B= \mathcal{C}$)
the bound \eqref{mb1}
 reduces to
\begin{align} \label{d211}
\begin{aligned}
 d_2(f_n, B)
&  \leq \frac{\sqrt{3}}{2} \bigg( \int_{\R^5} \|
D^2_{x,y}f_n(r_1)\|_{L^4(\Omega)}
\| D^2_{z,y}f_n(r_1)\|_{L^4(\Omega)} \\
&\qquad\qquad\qquad\times \| D_{x}f_n(r_2)\|_{L^4(\Omega)}
\| D_{z}f_n(r_2)\|_{L^4(\Omega)}dxdydz \nu(dr_1) \nu(dr_2) \bigg)^{\frac12}.
\end{aligned}
\end{align}

First, it is easy to get
\begin{align}\label{Dzfn0}
D_z f_n(r) = \frac{1}{\sqrt{n}} \sum_{j=1}^n
\Big[  \ind_{[2j, 2j+1)  }(z)  \tau(rw_j^{(0)})      + \ind_{[2j+1, 2j+2)
}(z)   rw_j \tau'(r w_j^{(0)}) \Big];
\end{align}

\noi
see also \eqref{c_chain}. Then, the fourth moment of $D_z f_n(r)$ is equal to
\begin{align*}
\frac{1}{n^2} \sum_{j=1}^n
\Big[  \ind_{[2j, 2j+1)  }(z)  \E[  | \tau(rw_j^{(0)}) |^4 ]      + 3 r^4
\ind_{[2j+1, 2j+2)  }(z)    \E[ |  \tau'(r w_j^{(0)})  |^4 ] \Big],
\end{align*}

and thus,
\begin{align}  \label{Dzfn}
\begin{aligned}
\|D_z f_n(r) \|_{L^4(\Omega)}
& =  \frac{1}{\sqrt{n}} \sum_{j=1}^n
\Big[  \ind_{[2j, 2j+1)  }(z)  \| \tau(rw_j^{(0)}) \|_{L^4(\Omega)}    \\
 &\qquad\qquad + 3^{\frac14} |r|
 \ind_{[2j+1, 2j+2)  }(z)
 \|\tau'(r w_j^{(0)})\|_{L^4(\Omega)}\Big] \\
 &\leq  \frac{C (1 + |r|^\gamma)}{\sqrt{n}} \sum_{j=1}^n
 \big[ \ind_{[2j, 2j+1)  }(z)  +   |r|  \ind_{[2j+1, 2j+2)  }(z)  \big] \\
 &\leq  \frac{C (1 + |r|^{1+\gamma})}{\sqrt{n}} \sum_{j=1}^n
  \ind_{[2j, 2j+2)  }(z).
\end{aligned}
\end{align}

Similarly, we first get from \eqref{Dzfn0} that
\begin{align*}
D_x D_z f_n(r)
&= \frac{1}{\sqrt{n}} \sum_{j=1}^n
\Big[  \ind_{[2j, 2j+1)  }(z)\ind_{[2j+1, 2j+2)  }(x)  r  \tau'(rw_j^{(0)})
\\
&\qquad\qquad  + \ind_{[2j+1, 2j+2)  }(z) \ind_{[2j+1, 2j+2)  }(x)   r^2 w_j
\tau''(r w_j^{(0)})  \\
&\qquad\qquad  + \ind_{[2j, 2j+1)  }(x) \ind_{[2j+1, 2j+2)  }(z)  r \, \tau'(r
w_j^{(0)}) \Big],
\end{align*}
and thus we can get

 \noi
\begin{align}  \label{Dxyfn}
\begin{aligned}
\|D^2_{x,z} f_n(r) \|_{L^4(\Omega)}
&\leq  \frac{C(1+|r|^\gamma)}{\sqrt{n}}  \sum_{j=1}^n
\Big[  \ind_{[2j, 2j+1)  }(z)\ind_{[2j+1, 2j+2)  }(x)  |r|     \\
&\qquad\qquad  + \ind_{[2j+1, 2j+2)  }(z) \ind_{[2j+1, 2j+2)  }(x)   r^2   \\
&\qquad\qquad  + \ind_{[2j, 2j+1)  }(x) \ind_{[2j+1, 2j+2)  }(z)   |r| \Big]
\\
&\leq \frac{C(1+|r|^{2+\gamma})}{\sqrt{n}}
\sum_{j=1}^n \ind_{[2j, 2j+2)  }(z) \ind_{[2j, 2j+2)  }(x).
\end{aligned}
\end{align}
Together with \eqref{env2} and \eqref{nu4ga},
these estimates and Minkowski's inequality show that
$f_n\in\mathbb D^{2,4}(K)$.
Thus Theorem \ref{thm_imp1} applies.
Plugging the bounds \eqref{Dzfn}-\eqref{Dxyfn} into \eqref{d211},
we get
\begin{align*}
&\int_{\R^5} \| D^2_{x,y}f_n(r_1)\|_{L^4(\Omega)}
\| D^2_{z,y}f_n(r_1)\|_{L^4(\Omega)} \| D_{x}f_n(r_2)\|_{L^4(\Omega)}
\| D_{z}f_n(r_2)\|_{L^4(\Omega)}dxdydz \nu(dr_1) \nu(dr_2)  \\
&\leq \frac{C}{n^2} \bigg( \int_{\R^2} (1+ |r_1|^{2+\gamma})^2 (1 +
|r_2|^{1+\gamma})^2 \nu(dr_1) \nu(dr_2) \bigg)
 \int_{\R^3}\sum_{j=1}^n \ind_{\{x, y,z\in [2j, 2j+2)  \}}  dxdy dz  \\
 &\leq  C/n,
\end{align*}
where in the last step, we used the condition \eqref{nu4ga} for the finite
measure $\nu$.
Hence the desired bound \eqref{CLT_NN} follows immediately.
\qedhere

\end{proof}

\subsection{When \texorpdfstring{$\fH$}{H} is not an
\texorpdfstring{$L^2$}{L2} space} \label{SEC32}
Motivated by recent active research on the spatial statistics over stochastic
partial differential equations
(e.g., \cite{HNV20,  BNQSZ, NXZ22, BZ24, BZ25}),
we consider the case where the Hilbert space $\fH$ is generated by
the correlation kernel of a space-time Gaussian colored noise.
More precisely, let $\dot{W}$ be a space-time Gaussian noise on
$\R_+\times\R^d$
with correlation structure formally given by

\noi
\begin{align}\label{ga01}
\E\big[ \dot{W}(t, x) \dot{W}(s, y)\big] = \gamma_0(t-s) \gamma_1(x-y),
\end{align}
where $\gamma_0, \gamma_1$ describe the temporal correlation and  spatial
correlation respectively.
When they are Dirac delta functions, $\dot{W}$ reduces to the space-time white
noise.
In the following, we make a few assumptions on the correlation kernels:
\begin{itemize}

\item[(i)] $\gamma_0:\R\to[0,\infty]$ is  nonnegative-definite and  locally
integrable;

\item[(ii)] $\gamma_1:\R^d\to[0,\infty]$ is the Fourier transform of some
nonnegative
tempered measure $\mu$ on $\R^d$ (i.e., $\gamma_1(x) = \mathcal{F}\mu(x) =
\int_{\R^d} e^{-i x \cdot \xi} \mu(d\xi)$), satisfying Dalang's condition
\begin{align}\label{DC}
\int_{\R^d} \frac{1}{ 1 + |\xi|^2} \mu(d\xi) < \infty.
\end{align}

\end{itemize}
The noise $\dot{W}$ with correlation structure \eqref{ga01} can be viewed
as a formal derivative $\partial_t \partial^d_{x_1...x_d} W$ of a Gaussian
process $\{W(\phi)\}_\phi $.
For $\phi,\psi \in C^\infty_c(\R_+\times\R^d)$, $W(\phi)$ and $W(\psi)$ are
centered, jointly Gaussian
with
\begin{align}\label{ga01b}
\begin{aligned}
\E\big[ W(\phi) W(\psi) \big] &=  \int_{\R_+^2\times\R^{2d}} \phi(t, x)
\psi(s, y) \gamma_0(t-s) \gamma_1(x-y) dtdsdxdy \\
&=: \langle \phi, \psi \rangle_\fH,
\end{aligned}
\end{align}
This form is nonnegative definite.
We quotient $C^\infty_c(\R_+\times\R^d)$
by its nullspace and complete the resulting inner product
space to obtain a real separable Hilbert space $\fH$.
By density, $W$ extends to an isonormal Gaussian process
$W=\{W(\phi):\phi\in\fH\}$
over $\fH$: $W$ is a centered Gaussian family with
$\E\big[ W(\phi) W(\psi) \big]=\langle \phi, \psi \rangle_\fH$ for any $\phi,
\psi\in\fH$.
See, e.g., \cite[Chapter 1]{Nua06}.
Starting from this isonormal Gaussian process $W$, we can go through the
same construction of Hilbert-valued Wiener structure.
Since the Hilbert space $\fH$ may contain some generalized functions,
the Malliavin derivative $DF$ of a generic Malliavin-differentiable real
random variable
may not be a random function; however, this case would be excluded in the
following applications.

Now let us briefly describe the motivating example of the above Gaussian
process.
Consider the following parabolic Anderson model (i.e., stochastic linear heat
equation)
driven by a Gaussian colored noise with \eqref{ga01}:
\begin{align}\label{SHE}
\begin{cases}
 \partial_t u - \frac{1}{2} \Delta u = u \dot{W} \\
 u(0, x) = 1, \,\, \forall x\in\R^d,
\end{cases}
\end{align}
where $\Delta = \sum_{i=1}^d \partial^2_{x_i}$ denotes the Laplacian on the
spatial variable.
Under the above assumptions (i)-(ii) on the correlation kernels,
the above equation admits a unique random field solution $\{ u(t,x):
(t,x)\in\R_+\times\R^d\}$;
see \cite{HHNT15}. More precisely, letting
$\mathcal{F}_t$ denote the $\sigma$-algebra generated by random variables
$W(\phi)$
with $\phi\in C^\infty_c(\R_+\times\R^d)$
such that $\phi(s,\bul)=0$ for $s>t$,
then $\mathbf{F} : = \{ \mathcal{F}_t\}_{t\in\R_+}$ is a filtration generated
by
the noise $\dot{W}$. We say   $\{ u(t,x): (t,x)\in\R_+\times\R^d\}$
is a mild solution (or random field solution) to the equation \eqref{SHE}
if $u$ is $\mathbf{F}$-adapted, $u(t,x)\in L^2(\Omega)$ for every
$(t,x)\in\R_+\times\R^d$,
and it holds almost surely that
\begin{align*}
u(t,x) = 1 + \int_0^t \int_{\R^d} p_{t-s}(x-y) u(s, y) W(ds, dy),
\end{align*}
where
$
p_t(x) =  \frac{1}{(2\pi t)^{d/2}} e^{- \frac{|x|^2}{2t}}
$
is the heat kernel,
and
 the above stochastic integral is defined as a Skorohod integral
(i.e., the integrand  $(s, y) \mapsto p_{t-s}(x-y) u(s, y) \ind_{[0, t]}(s)$
is
in the domain $\dom(\dl)$). Due to the linearity of the equation \eqref{SHE},
one can formally iterate the above integral equation and
obtain a formal series, and the unique existence of the mild solution
is equivalent to the $L^2(\Omega)$-convergence of the said series.
Note also that due to the homogeneous nature of the Gaussian noise
and the constant initial data, the solution is clearly spatially stationary.
Note that the first iteration would give us the term
$
\int_0^t \int_{\R^d} p_{t-s}(x-y) W(ds, dy),
$
which is clearly a centered Gaussian random variable with variance
\begin{align*}
\int_0^t \int_0^t \int_{\R^{2d}} p_{t-s}(x-y)p_{t-s'}(x-y') \gamma_0(s-s')
\gamma_1(y-y')dsds' dydy',
\end{align*}
whose finiteness follows from Dalang's condition
\eqref{DC} and the local integrability of $\gamma_0$,
by taking the Fourier transform in the spatial variables
and applying Young's convolution inequality in time.
  In a recent work \cite{NXZ22}, the authors established a
  quantitative central limit theorem for the spatial integral of the solution
$u(t,x)$.
  Our aim for this subsection is to derive a suitable
    version of the functional second-order Gaussian Poincar\'e inequality,
    with which we can provide a rate of convergence
    for the corresponding functional central limit theorem.
  To limit the length of this paper, we only consider a   case considered
  in the work  \cite{NXZ22} (i.e., $\gamma_1\in L^1(\R^d)$) and mention that
  the same strategy should work easily for the stochastic wave equations
  in \cite{BNQSZ}.

  \smallskip

  Now let us state the abstract bound for the case where $\fH$ is generated
  by the correlation structure \eqref{ga01}
  and $K = L^2(E, \mathscr{B}(E), \nu)$.

 \begin{theorem}\label{thm_imp2}
Let $F=\{F(r):r\in E\}\in\mathbb D^{2,4}(K)$
be centered, with covariance operator $S_F$.
Assume that $DF(r)$ and $D^2F(r)$ admit jointly
measurable function representatives on
$\R_+\times\R^d$ and $(\R_+\times\R^d)^2$,
respectively, such that, for $\nu$-almost every $r$,
$|DF(r)|\in\fH$ and
$|D^2F(r)|\in\fH^{\otimes2}$ almost surely.
Let $Z$ be a centered Gaussian random variable on $K$ with covariance operator
$S$.
Then,

 \noi
\begin{align}\label{mb3}
\begin{aligned}
d_2(F, Z) &\leq  \frac{\sqrt{3}}{2} \sqrt{\mathcal{A}}
 + \frac{1}{2} \| S_F - S_Z \|_{\rm HS},
\end{aligned}
\end{align}

\noi
where

 \noi
\begin{align*}
\mathcal{A}:=&\int_{E^2} \nu(dr_1) \nu(dr_2) \int_{\R_+^6\times\R^{6d}} drdr'
dsds' d\theta d\theta' dzdz' dydy' dwdw'
 \\
&    \times \gamma_0(\theta-\theta') \gamma_0(s-s') \gamma_0(r-r')
\gamma_1(z-z') \gamma_1(w-w') \gamma_1(y-y') \\
&    \times \| D_{r,z}D_{\theta, w} F(r_1)\|_{L^4(\Omega)}
\| D_{s,y}D_{\theta', w'} F(r_1)\|_{L^4(\Omega)}
\| D_{r',z'} F(r_2)\|_{L^4(\Omega)}  \| D_{s',y'} F(r_2)\|_{L^4(\Omega)}.
\end{align*}

\end{theorem}
When $\gamma_0, \gamma_1$ are Dirac delta functions, $\fH$ reduces to the
$L^2(\R_+\times\R^d)$
and the expression $\mathcal{A}$ has  a simple form as in \eqref{mb1}.

\begin{proof}[Proof of Theorem \ref{thm_imp2}]
The assertion is immediate if $\mathcal A=\infty$.
Let $Z_F$ be a centered Gaussian random variable
with covariance $S_F$.
By \eqref{Gaussian_comp} and the triangle inequality,
it suffices to prove the bound when $S=S_F$.
The assumption $F\in\mathbb D^{2,4}(K)$
justifies the product rule and Poincar\'e inequality
used in \eqref{bound1}, \eqref{bound2a},
and \eqref{bound2b}, which give
\begin{align*}
d_2(F, Z) \leq  \frac{1}{2} \sqrt{2\,  \E\[  \big\|  \jb{D^2F, -DL^{-1}F}_\fH
\big\|_{\fH \otimes K^{\otimes2}}^2\]
+2\, \E\[\big\|\jb{DF, -D^2L^{-1}F}_\fH \big\|_{\fH \otimes K^{\otimes2}}^2\]
}
\end{align*}
with
\begin{align*}
 \big\|  \jb{D^2F, -DL^{-1}F}_\fH
\big\|_{\fH \otimes K^{\otimes2}}^2
&= \bigg\|  \int_0^\infty dt e^{-t}  \langle D^2F,  P_tDF \rangle_\fH
\bigg\|_{\fH \otimes K^{\otimes2}}^2 \\
&\leq  \int_0^\infty dt e^{-t} \int_{E^2} \nu(dr_1) \nu(dr_2) \big\| \langle
D^2F(r_1),  P_tDF(r_2) \rangle_\fH  \big\|^2_{\fH}.
\end{align*}

\noi
Then, using the same argument as in \cite[Appendix A.2, p.825]{BNQSZ}
(with $F(r_1), F(r_2)$ in place of $F, G$ therein),
we have

\noi
\begin{align*}
&\E\big[  \big\| \langle D^2F(r_1),  P_tDF(r_2) \rangle_\fH  \big\|^2_{\fH}
\big] \\
&\quad \leq \mathbf{T}_{r_1, r_2} :=
\int_{\R_+^6\times\R^{6d}} drdr' dsds' d\theta d\theta' dzdz' dydy' dwdw'
 \\
&\qquad   \times \gamma_0(\theta-\theta') \gamma_0(s-s') \gamma_0(r-r')
\gamma_1(z-z') \gamma_1(w-w') \gamma_1(y-y') \\
&  \qquad \times \| D_{r,z}D_{\theta, w} F(r_1)\|_{L^4(\Omega)}
\| D_{s,y}D_{\theta', w'} F(r_1)\|_{L^4(\Omega)}
\| D_{r',z'} F(r_2)\|_{L^4(\Omega)}  \| D_{s',y'} F(r_2)\|_{L^4(\Omega)} ,
 \end{align*}
which does not depend on $t$. Thus, we have
\begin{align*}
\E\big[ \big\|  \jb{D^2F, -DL^{-1}F}_\fH
\big\|_{\fH \otimes K^{\otimes2}}^2 \big]
\leq  \int_{E^2} \nu(dr_1) \nu(dr_2) \mathbf{T}_{r_1, r_2}.
\end{align*}
In the same way, we can obtain
\begin{align*}
\E\big[ \big\|  \jb{DF, -D^2L^{-1}F}_\fH
\big\|_{\fH \otimes K^{\otimes2}}^2 \big]
\leq  \frac{1}{2} \int_{E^2} \nu(dr_1) \nu(dr_2) \mathbf{T}_{r_1, r_2}.
\end{align*}
Hence, we get \eqref{mb3}.
\qedhere

\end{proof}

  \noi
 $\bul$ {\bf Application $\III$: Spatial statistics of SPDEs.}
In a pioneering  work \cite{HNV20},
 Huang, Nualart,
and Viitasaari initiated a research line of establishing
 (quantitative) central limit theorems for stochastic partial differential
 equations (SPDEs); more precisely in   \cite{HNV20},
 they considered one-dimensional stochastic nonlinear heat equations driven
 by the space-time Gaussian white noise and established that
 the spatial integrals of the solution
 \begin{align}\label{FRT}
 F_R(t) := \int_{|x|\leq R} [u(t,x)-1]dx
 \end{align}

 \noi
 admit
 Gaussian fluctuation as $R\to\infty$.
 Their results were partially motivated by the study of mathematical
intermittence
 of the parabolic Anderson model, and are closely related
 to the spatial ergodicity of the corresponding  SPDEs (see, e.g.,
\cite{CKNP21, NZ20ecp}).
 We refer interested readers to  \cite[page 4]{BHWXY25} for a
 summary of results in this direction.
In the following, we apply Theorem \ref{thm_imp2}
to obtain a quantitative Gaussian approximation
in $d_2^\dagger$ on $L^2([0,T])$.
The limiting covariance is the one identified in
\cite[Theorem 1.6]{NZ20ejp}, and the derivative estimates
used below are from \cite[Theorem 3.1]{NXZ22}.
The rate concerns approximation by a Gaussian process
with the same covariance; convergence to the fixed
limiting process also uses the covariance limits.
Here,  we consider the parabolic Anderson model \eqref{SHE}
 with space-time colored noise $\dot{W}$ as in \eqref{ga01}
 and constant initial condition
 such that the temporal, spatial correlation kernels $\gamma_0,\gamma_1$
satisfy
 the above assumptions (i)-(ii).  Under these assumptions,
 we record below a useful bound from \cite[Theorem 3.1]{NXZ22}:
 for any $p\in[2,\infty)$ and for almost every $(s,y), (s_1,y_1), (s_2,
y_2)\in[0,t]\times\R^d$,
 we have
 \begin{align}\label{bdd_NXZ}
 \begin{aligned}
 \| D_{s,y} u(t,x) \|_{L^p(\Omega)} & \leq C(t)  p_{t-s}(x-y), \\
  \| D_{s_1,y_1}D_{s_2, y_2} u(t,x) \|_{L^p(\Omega)} &
  \leq   C(t)  \big[ p_{t-s_1}(x-y_1) p_{s_1-s_2}(y_1-y_2) \ind_{\{s_1 > s_2
\}}  \\
&\qquad\quad  + p_{t-s_2}(x-y_2) p_{s_2-s_1}(y_2-y_1) \ind_{\{s_2 > s_1 \}}
\big] ,
 \end{aligned}
 \end{align}

 \noi
 where the constant $C(t)$ depends on $(t,p, \gamma_0,\gamma_1)$ and is
increasing in $t$.
We may choose a common nondecreasing $C(t)\geq1$
for the bounds with $p=2$ and $p=4$.
Only its local boundedness is used on a finite time interval.
These estimates also verify the Sobolev assumptions
of Theorem \ref{thm_imp2} on each finite time interval.
Indeed, the deterministic kernels on the right-hand side
of \eqref{bdd_NXZ} are, up to fixed factors,
the first two chaos kernels of $u(t,x)$.
Their $\fH^{\otimes j}$-norms, $j=1,2$,
are bounded uniformly for $t\leq T$ and $x\in\R^d$
by the chaos isometry and the uniform second-moment
bound from \cite{HHNT15}.
Since the covariance kernels are nonnegative,
Minkowski's inequality applied to \eqref{bdd_NXZ}
therefore gives
\begin{align*}
 \sup_{t\leq T}\sup_{x\in\R^d}
 \E\big[\|\,|D^ju(t,x)|\,\|_{\fH^{\otimes j}}^4\big]
 <\infty,\qquad j=1,2.
\end{align*}
The solution also has locally bounded fourth moments
\cite{HHNT15}.
Bochner integration and the closedness of the Malliavin
derivatives now yield
$F_R\in\mathbb D^{2,4}(L^2([0,T]))$.
The same argument gives the required absolute-function
memberships for its derivatives.

 \smallskip

 Now we are ready to state our result on the quantitative Gaussian process
 approximation of the process $\{F_R(t): t\in[0,T]\}$ for any finite $T > 0$.

 \begin{theorem} \label{APP_SPDE}
 Let the above assumptions prevail.
 Suppose that $\gamma_0$ is nontrivial, i.e.,
 $\int_0^a\int_0^a \gamma_0(t-t')dtdt' >0$ for all $a > 0$.
 Then, for any   $T\in(0,\infty)$,
 $F_R=\{F_R(t): t\in[0,T]\}$ is a random element
 in $K:=L^2([0,T],dt)$.
 Assume that $0 <   \| \gamma_1\|_{L^1(\R^d)} <\infty$. Then, we have

 \noi
\begin{align}\label{bdd_361}
\begin{aligned}
 d_2^\dagger(R^{-\frac d2}F_R,\mathcal G_R)
 &=d_2(R^{-\frac d2}F_R,\mathcal G_R)\\
 &\leq CR^{-\frac d2},
\end{aligned}
\end{align}

 \noi
 where $\mathcal{G}_{R}$ is a centered Gaussian process with the same
covariance operator as $R^{-\frac{d}{2}}F_R$.
Let $\mathcal{G}$
 be a continuous centered Gaussian process with covariance structure
 \begin{align}\label{cov_G1}
\mathcal{C}_{\infty}(t,s):=
\E\big[\mathcal{G}(t)\mathcal{G}(s)\big]
=\omega_d\int_{\R^d}
{\rm Cov}(u(t,z),u(s,0))dz
 \end{align}

 \noi
 with $\omega_d$ the volume of a unit ball in $\R^d$
 and $(s,t)\in[0,T]^2$. Then,
  $d_2^\dagger(R^{-\frac d2}F_R,\mathcal G)\to0$
  as $R\to\infty$.
  In particular, $W_2(R^{-\frac d2}F_R,\mathcal G)\to0$
  and $R^{-\frac d2}F_R$ converges weakly to $\mathcal G$
  in $L^2([0,T])$.

%

 \end{theorem}

\begin{proof}

We know from \cite{HHNT15} that the solution   to \eqref{SHE}
uniquely exists with
\begin{align*}
\sup_{t\leq T} \sup_{x\in\R^d} \E[ u(t,x)^2 ] <\infty.
\end{align*}
As a result, we can see easily that  $\E\big[ \| F_R \|_{K}^2 \big] <\infty$,
and in particular,
we have $F_R\in K$ almost surely.
Note
that the covariance structure \eqref{cov_G1}
is the limiting covariance structure of the process   $R^{-\frac{d}{2}}F_R$
and
actually coincides with the one in \cite[Theorem 1.6]{NZ20ejp} (see equation
(1.13) therein).
To see, we write
\begin{align*}
\mathcal{C}_{R}(t,s):&=R^{-d}  \E[ F_R(t)F_R(s)]\\
&=R^{-d}  \int_{|x|\leq R} \int_{|y| \leq R} \Cov(u(t, x), u(s, y) )dxdy \\
&= R^{-d}  \int_{|x|\leq R} \int_{|y| \leq R}  \Cov(u(t, x-y), u(s, 0) )dxdy
\quad\text{by  stationarity;} \\
&=\omega_d   \int_{\R^d}   \Cov(u(t, z), u(s, 0) )   \frac{ {\rm Vol}\{x:
|x|\leq R, |x-z|\leq R\} }{{\rm Vol}\{x: |x|\leq R\} }   dz \\
&\xrightarrow{R\to\infty} \omega_d   \int_{\R^d}   \Cov(u(t, z), u(s, 0) )
dz,
\end{align*}
provided
\begin{align}\label{cond_G1}
\int_{\R^d}   | \Cov(u(t, z), u(s, 0) ) | dz <\infty
\end{align}
due to dominated convergence.  Indeed, the condition \eqref{cond_G1} can be
verified as follows.
Using $\Cov(u(t, z), u(s, 0) )  = \E[ \langle D u(t,z), -DL^{-1} u(s, 0)
\rangle_\fH ]$
and the representation \eqref{rep4},\footnote{in particular its consequence
$\| -D_{r, y} L^{-1}F \|_{L^p(\Omega)} \leq \| D_{r, y} F\|_{L^p(\Omega)} $
for any real-valued Malliavin differentiable random variable $F$ such that
$\| D_{r, y} F\|_{L^p(\Omega)}$ is finite.}
we deduce from \eqref{bdd_NXZ} that
\begin{align*}
&  | \Cov(u(t, z), u(s, 0) ) |  \\
 & \leq \E\bigg[  \int_{\R_+^2}  \int_{\R^{2d}} \gamma_0(r-r')
  \gamma_1(y-y') | D_{r, y} u(t,z) |  | D_{r', y'}L^{-1} u(s,0) |  drdr' dy
dy' \bigg]\\
  &\leq  \int_{\R_+^2}  \int_{\R^{2d}} \gamma_0(r-r')
  \gamma_1(y-y') \| D_{r, y} u(t,z) \|_{L^2(\Omega)}   \| D_{r', y'} u(s,0)
\|_{L^2(\Omega)}  drdr' dy dy'  \\
  &\leq C(t) C(s)  \int_{0}^t \int_0^s drdr' \gamma_0(r-r') \int_{\R^d}
p_{t-r}(z-y) p_{s-r'}(y') \gamma_1(y-y')dydy',
\end{align*}
which is clearly integrable in $z$ over $\R^d$ using $\gamma_1\in L^1(\R^d)$.
That is, the condition \eqref{cond_G1} is verified.  Using exactly the same
argument via dominated convergence,
one can easily see that
\begin{align}\label{cov_lim1}
\lim_{R\to+\infty} \int_{[0,T]^2} \big| \mathcal{C}_{\infty}(t,s) -
\mathcal{C}_{R}(t,s) \big|^2 dtds =0.
\end{align}

\noi
That is, $\mathcal C_R$ converges to $\mathcal C_\infty$
in Hilbert-Schmidt norm, and hence
$d_2(\mathcal G_R,\mathcal G)\to0$ by \eqref{Gaussian_comp}.
To control the second moments, put
\begin{align}\label{time_cov}
 J_0(t):=\int_0^t\int_0^t\gamma_0(r-r')drdr'
 =2\int_0^t(t-v)\gamma_0(v)dv.
\end{align}
The preceding covariance estimate gives, for $t\leq T$,
\begin{align}\label{SPDE_diag}
 0\leq\mathcal C_R(t,t)
 &\leq\omega_d\|\gamma_1\|_{L^1(\R^d)}C(t)^2J_0(t)
 \notag\\
 &\leq\omega_d\|\gamma_1\|_{L^1(\R^d)}C(T)^2J_0(T).
\end{align}
Since $\mathcal C_R(t,t)\to\mathcal C_\infty(t,t)$,
dominated convergence yields
\begin{align}\label{SPDE_trace}
 \E[\|R^{-d/2}F_R\|_K^2]
 &=\int_0^T\mathcal C_R(t,t)dt
 \notag\\
 &\longrightarrow\int_0^T\mathcal C_\infty(t,t)dt
 =\E[\|\mathcal G\|_K^2].
\end{align}
The same identity holds with $\mathcal G_R$
in place of $R^{-d/2}F_R$.
It follows that $d_2^\dagger(\mathcal G_R,\mathcal G)\to0$.
Moreover, the matching covariances imply
$d_2^\dagger(R^{-d/2}F_R,\mathcal G_R)
=d_2(R^{-d/2}F_R,\mathcal G_R)$.
Thus, once \eqref{bdd_361} is proved,
the triangle inequality and Proposition \ref{prop_d2mom}
give all the convergence assertions.
Now let us apply our bound \eqref{mb3} for $K$-valued random variable $F_R$:
\begin{align*}
d_2( R^{-\frac{d}{2}}F_R, \mathcal{G}_R ) \leq \frac{\sqrt{3}}{2}
\sqrt{\mathcal{A}}
\end{align*}
with
\begin{align*}
\mathcal{A}
&= R^{-2d}  \int_{[0,T]^2}dr_1 dr_2 \int_{\R_+^6 \times\R^{6d}}
drdr' dsds' d\theta d\theta' dzdz' dydy' dwdw'
 \\
&    \times \gamma_0(\theta-\theta') \gamma_0(s-s') \gamma_0(r-r')
\gamma_1(z-z') \gamma_1(w-w') \gamma_1(y-y') \\
&    \times \| D_{r,z}D_{\theta, w} F_R(r_1)\|_{L^4(\Omega)}
\| D_{s,y}D_{\theta', w'} F_R(r_1)\|_{L^4(\Omega)}
\| D_{r',z'} F_R(r_2)\|_{L^4(\Omega)}  \| D_{s',y'} F_R(r_2)\|_{L^4(\Omega)}.
\end{align*}

Then, proceeding as in \cite[Section 3.1.1]{NXZ22}
and using \eqref{bdd_NXZ}, we obtain
\begin{align}\label{forRM1}
 \mathcal{A}
 \leq R^{-2d}\int_{[0,T]^2}
 C(r_1)^2C(r_2)^2
 \mathcal{A}^\ast_{r_1,r_2}\,dr_1dr_2,
\end{align}
where
\begin{align*}
 \mathcal{A}^\ast_{r_1,r_2}
 :=&\int_{[0,r_1]^4}drdsd\theta d\theta'
 \int_{[0,r_2]^2}dr'ds' \gamma_0(r-r')\gamma_0(s-s')
 \gamma_0(\theta-\theta')\\
 &\times\int_{\R^{6d}}dzdz'dydy'dwdw'\,
 \gamma_1(z-z')\gamma_1(y-y')
 \gamma_1(w-w')\\
 &\times\int_{B_R^4}d\pmb{x}_4\,
 \big[
 p_{r_1-r}(x_1-z)p_{r-\theta}(z-w)
 +p_{r_1-\theta}(x_1-w)p_{\theta-r}(w-z)
 \big]\\
 &\times\big[
 p_{r_1-s}(x_2-y)p_{s-\theta'}(y-w')
 +p_{r_1-\theta'}(x_2-w')p_{\theta'-s}(w'-y)
 \big]\\
 &\times p_{r_2-r'}(x_3-z')
 p_{r_2-s'}(x_4-y').
\end{align*}
Here $B_R=\{x\in\R^d:|x|\leq R\}$,
$d\pmb{x}_4=dx_1dx_2dx_3dx_4$,
and heat kernels with nonpositive time parameter
are taken to be zero.

Expanding the two brackets gives four terms.
For each term, extend the domains of
$x_1,x_2,x_4$ to $\R^d$.
Then integrate with respect to
$dx_1$, $dx_2$, $dx_4$, $dy'$, $dy$,
$dw'$, $dw$, $dz$, $dz'$, and $dx_3$,
in this order.
Using $\int_{\R^d}p_t(x)\,dx=1$ for $t>0$,
we bound each spatial integral by
$\omega_dR^d\|\gamma_1\|_{L^1(\R^d)}^3$.

Set $M=\max\{r_1,r_2\}$.
Extending the six time integrals to $[0,M]$
and using
\begin{align*}
 \int_0^M\int_0^M\gamma_0(u-v)\,dudv
 &=2\int_0^M(M-v)\gamma_0(v)\,dv
 \leq 2M\int_0^M\gamma_0(v)\,dv,
\end{align*}
we obtain
\begin{align}\label{forRM2}
 \mathcal{A}^\ast_{r_1,r_2}
 \leq 4\omega_dR^d
 \|\gamma_1\|_{L^1(\R^d)}^3
 \left(2M\int_0^M\gamma_0(v)\,dv\right)^3.
\end{align}
Since $\gamma_0$ is locally integrable,
\eqref{forRM1} and \eqref{forRM2}
give $\mathcal{A}\leq CR^{-d}$
for some finite constant $C$ independent of $R$.
Hence \eqref{bdd_361} follows. \qedhere

\end{proof}

\begin{remark}\label{rem:inftime}\rm
In Theorem \ref{APP_SPDE}, we considered the integral
process over a finite interval $[0,T]$.
In fact, it is possible to consider
$F_R\in L^2(\R_+,d\nu)$ with $\nu$ satisfying
certain integrability at infinity.
Using $C(r)=ae^{br}$ with $a,b>0$
and using also the bound \eqref{forRM2},
we have (instead of \eqref{forRM1}),
\begin{align*}
 \mathcal A
 &\leq 32R^{-2d}
 \int_0^\infty\nu(dr_2)
 \int_0^{r_2}\nu(dr_1)\,
 C(r_1)^2C(r_2)^2
 \mathcal A^\ast_{r_1,r_2}\\
 &\leq 32R^{-2d}a^4
 \int_0^\infty\nu(dr_2)
 \int_0^{r_2}\nu(dr_1)\,
 e^{2b(r_1+r_2)}
 \omega_dR^d\|\gamma_1\|_{L^1(\R^d)}^3
 \bigg(2\int_0^{r_2}\gamma_0(v)\,dv\bigg)^3.
\end{align*}
For example, taking $\nu(dr)=\psi(r)\,dr$
with $\psi(r)\leq e^{-3br}$ such that
\begin{align*}
 \int_0^\infty\nu(dr)\,e^{2br}
 \bigg(\int_0^r\gamma_0(v)\,dv\bigg)^3
 <\infty,
\end{align*}
we get $\mathcal A\leq CR^{-d}$,
i.e., the bound \eqref{bdd_361} holds for
$K=L^2(\R_+,d\nu)$ in place of
$K=L^2([0,T],dt)$.
Under the same conditions, the limit
\eqref{cov_lim1} can be replaced by
\begin{align*}
 \lim_{R\to+\infty}
 \int_{\R_+^2}
 \big|\mathcal C_\infty(t,s)
       -\mathcal C_R(t,s)\big|^2
 \nu(dt)\nu(ds)=0,
\end{align*}
and thus we have $\{F_R(t):t\in\R_+\}$
converges to $\mathcal G$ with respect to
the $d_2$ distance on $L^2(\R_+,d\nu)$,
as $R\to+\infty$.
\end{remark}
  \bigskip

  \ackno{\rm
While writing this work at the University of Naples ``Federico II'', A.V. has
been supported by the co-financing of the European
Union--FSE--REACT--EU, PON Research
and Innovation 2014-2020, DM 1062/2021. A.V. is a member of INdAM-GNAMPA.
The authors are grateful to Domenico Marinucci for his enlightening
discussions at the very early stage of this work.
G.Z. also  thanks Domenico Marinucci for his great hospitality during a visit
at Tor Vergata in 2023 October, when this work was initiated.

\it Since the first version of this paper appeared on arXiv, we have not made any change until A.V. visited G.Z. at Boston University in October 2026, when we re-examined the weak convergence issue for our $d_2$ bounds. We added the $d_2^\dagger$ distance \eqref{d2dag} that metrizes the $W_2$ convergence and thus weak convergence {\rm(}see Proposition \ref{prop_d2mom}  {\rm)}. This is motivated by recent work \cite{DZ24, DZ25}.
We used GPT-6 for the final proofreading and adversarial check, and we made minor changes accordingly. The  authors are responsible for any mistake. We would also like to thank Giovanni Peccati for encouraging us to deal with the weak convergence issue for $d_2$ bounds. }

\end{document}